\documentclass[envcountsect,smallcondensed]{svjour3}

\usepackage[colorlinks=true, urlcolor=black, citecolor=black, linkcolor=black, hyperfootnotes=true]{hyperref}
\usepackage{amsmath}
\usepackage{amssymb}
\usepackage{mathtools}
\usepackage{mathrsfs}
\usepackage{tikz}
\usepackage{graphicx}

\newcommand*{\C}{\mathrel{\mathsf{C}}}
\newcommand*{\D}{\mathrel{\mathsf{D}}}

\newcommand*{\pp}{\mathrel{\prec\hspace{-4pt}\prec}}

\numberwithin{equation}{section}

\spnewtheorem{thm}{Theorem}[section]{\bfseries}{\itshape}
\spnewtheorem{prp}[thm]{Proposition}{\bfseries}{\itshape}
\spnewtheorem{cor}[thm]{Corollary}{\bfseries}{\itshape}
\spnewtheorem{dfn}[thm]{Definition}{\bfseries}{\rmfamily}
\spnewtheorem{xpl}[thm]{Example}{\bfseries}{\rmfamily}
\spnewtheorem{rmk}[thm]{Remark}{\bfseries}{\rmfamily}

\author{Tristan Bice \and Charles Starling}
\institute{Tristan Bice \at Institute of Mathematics of the Czech Academy of Sciences, \v{Z}itn\'a 25, Prague, Czech Republic \\ \email{tristan.bice@gmail.com}
\and Charles Starling \at Carleton University\\ \email{cstar@math.carleton.ca}}

\title{Locally Hausdorff Tight Groupoids Generalised\thanks{Tristan Bice is supported by the GA\v{C}R project EXPRO 20-31529X and RVO: 67985840.\\  Charles Starling is supported by a Carleton University internal research grant.}}

\begin{document}

\begin{abstract}
We extend Exel's ample tight groupoid construction to non-ample groupoids, even in the general locally Hausdorff case.
\end{abstract}

\maketitle

\section*{Introduction}

\subsection*{Motivation}

Point-free topology traces its origins to the classic Stone and Wallman dualities from the late 1930's (see \cite{Stone1936} and \cite{Wallman1938}), the basic idea being that certain topological spaces can be encoded by more algebraic/order theoretic structures.  In one direction, one typically takes a sublattice of basic open sets of a given space to obtain such an order structure.  In the other direction, one usually takes certain filters in a given lattice to obtain the points of a corresponding topological space.

More recently, over the past decade or so, non-commutative extensions of these dualities have to come the fore.  Here the idea is that when the topological space carries some compatible groupoid structure (such in the \'etale groupoids commonly arising in dynamical systems and operator algebras) this can be encoded with some extra semigroup structure.  This is particular true for ample/totally disconnected groupoids, where Lawson's non-commutative Stone duality (see \cite{Lawson2012}) and Exel's tight groupoid construction (see \cite{Exel2008}) have played a big role.

Our previous work in \cite{BiceStarling2018} and \cite{BiceStarling2020HTight} has focused on extending these further to non-ample locally compact \'etale groupoids.  In particular, in \cite{BiceStarling2020HTight} we extended Exel's tight groupoid construction to non-ample groupoids in the Hausdorff case.  When Exel's construction produces a non-Hausdorff groupoid $G$, our construction in \cite{BiceStarling2020HTight} still applies but instead produces a kind of `Hausdorffification' of $G$.  On the one hand, this is advantageous because Hausdorff spaces are generally easier to work with.  But on the other hand, it would be nice if we could construct more general locally Hausdorff spaces and groupoids in a similar way.  The goal of this paper is to show that this can indeed be achieved by `localising' our previous construction in an appropriate manner.

\subsection*{Outline}

We start in \S\ref{ABP} with a brief review pseudobases and an investigation of stronger bi-pseudobases.  In \S\ref{LTF}, we move on to examine the locally tight spectrum of locally tight filters in local bi-pseudobases.  The key result here is Theorem \ref{HausdorffEquivalents}, which says the locally tight spectrum of $P$ is Hausdorff iff it agrees with the tight spectrum iff $P$ is a bi-pseudobasis.  As mentioned after the proof, this generalises the Hausdorff characterisation of Exel's tight groupoid from \cite[Theorem 3.16]{ExelPardo2016}.

Next we show in \S\ref{CLBP} that abstract local bi-pseudobases can indeed be represented as appropriately defined concrete local bi-pseudobases.  Conversely, Theorem \ref{Recovery} says that we can always recover a space from a from a given concrete local bi-pseudobasis via the locally tight spectrum, thus providing a duality which encompasses various Stone duality extensions.

In \S\ref{OG} we discuss ordered groupoids in a slightly more general context than usual (with transitive relations rather than preorders or partial orders) and how they relate to inverse semigroups.  We then examine cosets and their groupoid structure in \S\ref{TheCosetGroupoid}, returning to focus on locally tight filters in \S\ref{LTG}.  The key results here are Theorem \ref{LGetale}, which shows that the locally tight spectrum is an \'etale groupoid, and Theorem \ref{GroupoidRecovery}, which shows conversely how to recover an \'etale groupoid, generalising \cite[Theorem 4.8]{Exel2010}.

Let us point out that the first half (\S\ref{ABP}-\S\ref{CLBP}) is purely order theoretic/topological, as opposed to the more algebraic second half (\S\ref{OG}-\S\ref{LTG}).  This continues the order theoretic approach initiated in \cite{Lenz2008} and maintained in several subsequent papers (see \cite{LawsonLenz2013} and \cite{LawsonMargolisSteinberg2013}).  We hope this not only clarifies the topological aspect of the locally tight groupoid construction but also makes it more accessible to topologists.

\newpage

The following diagram summarises how this paper fits into previous Stone duality extensions, going down in increasing generality.\\
\\
\\
\\

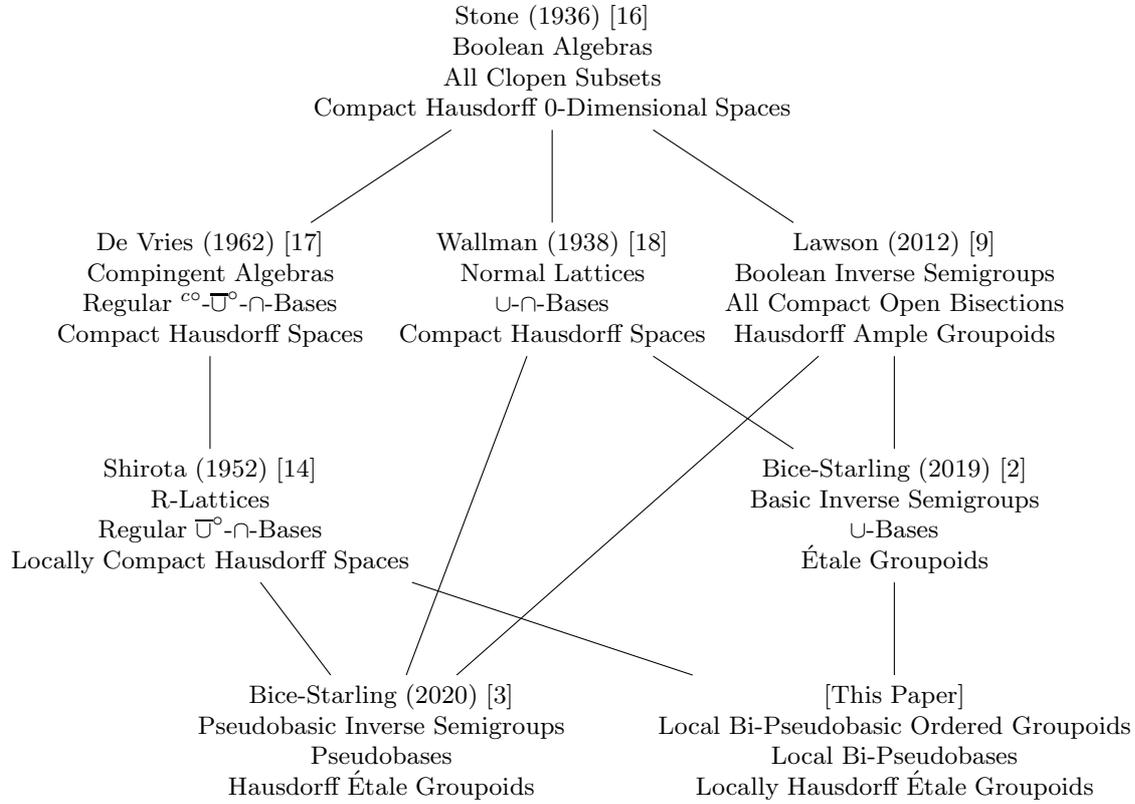
\hspace{-25pt}\begin{tikzpicture}[scale=3]
    \node (Stone) at (0,0) [align=center]{Stone (1936) \cite{Stone1936}\\ Boolean Algebras\\ All Clopen Subsets\\ Compact Hausdorff $0$-Dimensional Spaces};
	  \node (De Vries) at (-1.5,-1) [align=center]{De Vries (1962) \cite{DeVries1962}\\ Compingent Algebras\\ Regular ${}^{c\circ}$-$\overline{\cup}^\circ$-$\cap$-Bases\\ Compact Hausdorff Spaces};
		\node (Wallman) at (0,-1) [align=center]{Wallman (1938) \cite{Wallman1938}\\ Normal Lattices\\ $\cup$-$\cap$-Bases\\ Compact Hausdorff Spaces};
		\node (Lawson) at (1.5,-1) [align=center]{Lawson (2012) \cite{Lawson2012}\\ Boolean Inverse Semigroups\\ All Compact Open Bisections\\ Hausdorff Ample Groupoids};
		\node (Shirota) at (-1.5,-2) [align=center]{Shirota (1952) \cite{Shirota1952}\\ R-Lattices\\ Regular $\overline{\cup}^\circ$-$\cap$-Bases\\ Locally Compact Hausdorff Spaces};
		\node (BiceStarling) at (1.5,-2) [align=center]{Bice-Starling (2019) \cite{BiceStarling2018}\\ Basic Inverse Semigroups\\ $\cup$-Bases\\ \'Etale Groupoids};
		\node (BiceStarling2) at (-.75,-3) [align=center]{Bice-Starling (2020) \cite{BiceStarling2020HTight}\\ Pseudobasic Inverse Semigroups\\ Pseudobases\\ Hausdorff \'Etale Groupoids};
		\node (BiceStarling3) at (1.5,-3) [align=center]{[This Paper]\\ Local Bi-Pseudobasic Ordered Groupoids\\ Local Bi-Pseudobases\\ Locally Hausdorff \'Etale Groupoids};
		
		\draw (Stone)--(De Vries)--(Shirota)--(BiceStarling2);
		\draw (Shirota)--(BiceStarling3);
		\draw (Stone)--(Wallman)--(BiceStarling)--(BiceStarling3);
		\draw (Wallman)--(BiceStarling2);
		\draw (Lawson)--(BiceStarling2);
		\draw (Stone)--(Lawson)--(BiceStarling);
\end{tikzpicture}

\newpage

\section{Abstract Bi-Pseudobases}\label{ABP}

In \cite{BiceStarling2020HTight}, we exhibited a kind of duality between abstract and concrete `pseudobases'.  A concrete pseudobasis just is a family of open sets satisfying several basis-like properties \textendash\, see \cite[Definition 2.4]{BiceStarling2020HTight}.  In particular, every basis of a locally compact space is indeed a concrete pseudobases, but in general the converse is false.  The motivation for considering more general pseudobases is that they still provide enough structure to recover the space (at least in the Hausdorff case), while at the same time admitting a simpler order theoretic characterisation in terms of the cover relation $\C$, which naturally leads to the notion of an abstract pseudobasis.

To obtain the desired locally Hausdorff extension of our previous results, we need to consider a stronger variant of pseudobases, namely `bi-pseudobases', as well as their natural `localisation'.  We start by considering the abstract structures and follow up with their concrete counterparts below in \S\ref{CLBP}.

First let us recall some general notational conventions from \cite{BiceStarling2020HTight}.

\begin{dfn}
For any $Q\subseteq P$ and $\prec\ \subseteq P\times P$, we define
\[Q^\prec=\{p\in P:Q\ni q\prec p\}.\]
We extend $\prec$ to the `refinement' relation on $\mathscr{P}(P)=\{Q:Q\subseteq P\}$ by defining
\[Q\prec R\qquad\Leftrightarrow\qquad Q\subseteq R^\succ\qquad\Leftrightarrow\qquad\forall q\in Q\ \exists r\in R\ (q\prec r).\]
From $\prec$ we define relations $\D$ and $\C$ on $\mathscr{P}(P)$ by
\begin{align}
\label{DenseCover}\tag{Dense Cover}Q\D R\qquad&\Leftrightarrow\qquad\forall q\prec Q\ \exists r\prec R\ (r\prec q).\\
\label{CompactCover}\tag{Compact Cover}Q\C R\qquad&\Leftrightarrow\qquad\exists\text{ finite }F\ (Q\D F\prec R).
\end{align}
\end{dfn}

\begin{rmk}
Here we are imagining that $P$ is a collection of open sets of a space $X$ and $\prec$ is `compact containment', i.e. $p\prec q$ means that we have a compact subset $C$ of $X$ with $p\subseteq C\subseteq q$.  The compact cover relation $\C$ is meant to extend compact containment to unions, i.e. $Q\C R$ should mean that we have a compact subset $C$ of $X$ with $\bigcup Q\subseteq C\subseteq\bigcup R$.
\end{rmk}

We refer the reader to \cite[Proposition 2.14]{BiceStarling2020HTight} for basic properties of $\D$ and $\C$.

\begin{dfn}\label{PseudoDefinitions}
Assume $\prec$ is a transitive round relation on $P$, i.e. for all $p\in P$,
\[\tag{Transitive + Round}p^{\succ\succ}\subseteq p^\succ\neq\emptyset.\]
We then call $(P,\prec)$ an \emph{abstract (local) (bi-)pseudobasis} if, for all $p,q\in P$,
\begin{align}
\label{Pseudobasis}\tag{Pseudobasis}&\forall p'\prec p&p'^\succ&\C p^\succ.\\
\label{BiPseudobasis}\tag{Bi-Pseudobasis}\forall p'\prec p\ &\forall q'\prec q&(p'^\succ\cap q'^\succ)&\C(p^\succ\cap q^\succ).\\
\label{LocalBiPseudobasis}\tag{Local Bi-Pseudobasis}\forall r,s\preceq p\ \forall r'\prec r\ &\forall s'\prec s&(r'^\succ\cap s'^\succ)&\C(r^\succ\cap s^\succ).
\end{align}
Here $\preceq$ denotes the \emph{lower preorder} defined from $\prec$ by
\[p\preceq q\qquad\Leftrightarrow\qquad p^\succ\subseteq q^\succ.\]
\end{dfn}

\begin{rmk}
Note $\prec$ itself is a preorder iff $\prec\ =\ \preceq$.  In this case it suffices to consider $p'=p$ and $q'=q$ above, in which case \eqref{BiPseudobasis} is called the `weak meet condition' and the resulting structure a `weak semilattice' in \cite[\S1 Preliminaries]{LawsonLenz2013} and \cite{Steinberg2010} (in the comments just before Theorem 5.17).

Indeed, up until now, the tight groupoid construction has only been considered for the canonical partial order $\leq$ on an inverse semigroup.  However, in more general *-semigroups, the canonical order is only a transitive relation and it is vital to be able to handle these in the same way if we want to unify the tight groupoid construction with the Weyl groupoid construction \textendash\, see \cite{Bice2019a}.
\end{rmk}

From Definition \ref{PseudoDefinitions} we immediately see that
\[\eqref{BiPseudobasis}\qquad\Rightarrow\qquad\eqref{LocalBiPseudobasis}\qquad\Rightarrow\qquad\eqref{Pseudobasis}.\]
All these conditions are also immediately seen to be `downwards hereditary', i.e. if $P$ is an abstract (local) (bi-)pseudobasis and $Q^\succ\subseteq Q\subseteq P$ then $Q$ is also an abstract (local) (bi-)pseudobasis.

Also note \eqref{Pseudobasis} is saying every cover has a `shrinking', i.e.
\begin{align}
\label{Shrinking}\tag{Shrinking}p'\prec p\qquad&\Rightarrow\qquad\exists\text{ finite }F\ (p'\C F\prec p)\\
\intertext{In fact, as shown in \cite[Proposition 2.3]{BiceStarling2020HTight}, \eqref{Shrinking} can be strengthened to}
\label{CShrinking}\tag{$\C$-Shrinking}Q'\C Q\qquad&\Rightarrow\qquad\exists\text{ finite }F\ (Q'\C F\prec Q),
\end{align}
i.e. $Q^\succ\C R^\succ$ whenever $Q\C R$.

As the name suggests, any abstract pseudobasis $P$ can be realised as a concrete pseudobasis of some locally compact Hausdorff space, the points of the space being the tight subsets of $P$.  To define these, first let $\widehat{Q}$ (which corresponds to $Q^\wedge$ in \cite{Lawson2012}) denote the \emph{formal meet} of any $Q\subseteq P$ given by
\[\widehat{Q}=\bigcap_{q\in Q}q^\succ\]

\begin{dfn}\label{TightDefinition}
We call $T\subseteq P$ \emph{tight} if $T=T^\prec$ and $\widehat{F}\not\C P\setminus T$, for all finite $F\subseteq T$.
\end{dfn}

In other words, $T$ is tight iff $T$ is round ($T\subseteq T^\prec$) and, for all finite $F,G\subseteq P$,
\begin{equation}\label{Tight-Round}
F\subseteq T\quad\text{and}\quad G\prec P\setminus T\qquad\Rightarrow\qquad\widehat{F}\cap G^\perp\neq\emptyset.
\end{equation}
Here $G^\perp=\{q\in Q:G\perp q\}$ where $\perp$ denotes the disjoint relation defined by
\[R\perp Q\qquad\Leftrightarrow\qquad R^\succ\cap Q^\succ=\emptyset.\]

In most constructions of topological spaces from order structures, the points are defined to be certain kinds of filters, i.e. subsets $T$ satisfying
\begin{align}
\tag{Filter}t,u\in T\qquad&\Leftrightarrow\qquad\exists s\in T\ (s\prec t,u)\\
\intertext{Equivalently, $T$ is a filter iff $T$ is an \emph{up-set}, i.e. $T^\prec\subseteq T$, and \emph{(down-)directed}, i.e.}
\tag{Directed}t,u\in T\qquad&\Rightarrow\qquad\exists s\in T\ (s\prec t,u).\\
\intertext{This includes Exel's tight spectrum and the more general locally tight spectrum we consider in the present paper.  However, it is important to note that the tight subsets defined above need only satisfy the weaker condition $T=T^\prec$, i.e.}
\tag{Round Up-Set}t\in T\qquad&\Leftrightarrow\qquad\exists s\in T\ (s\prec t).
\end{align}
Indeed, it is crucial to consider these more general non-filter round up-sets when dealing with general pseudobases, as we did in \cite{BiceStarling2020HTight}.

However, in bi-pseudobases, tight subsets will automatically be filters.  Moreover, this fact characterises bi-pseudobases among pseudobases.

\begin{thm}\label{MeetPreservingEquivalents}
If $(P,\prec)$ is an abstract pseudobasis, the following are equivalent.
\begin{align}
\label{MeetPreserving}&&P\text{ is an abstract bi-pseudobasis}.&\\
\label{Tight=>Filter}\forall\ T\subseteq P&&T\text{ is tight}\qquad\Rightarrow\qquad& T\text{ is a filter}.\\
\label{FiniteMeetPreserving}\forall\text{ finite }F,G\subseteq P&&\emptyset\neq G\succ F\qquad\Rightarrow\qquad&\widehat{F}\C\widehat{G}.\\
\label{CMeetPreserving}\forall\ Q,R,S\subseteq P&&Q\C R\quad\text{and}\quad Q\C S\qquad\Rightarrow\qquad& Q\C R^\succ\cap S^\succ.
\end{align}
\end{thm}

\begin{proof}
Consider the following statement: for all non-empty finite $\Phi,\Psi\subseteq\mathscr{P}(P)$,
\begin{equation}\label{FiniteCMeetPreserving}
\forall R\in\Psi\ \exists Q\in\Phi\ (Q\C R)\ \Rightarrow\ {\textstyle\bigcap\limits_{Q\in\Phi}}Q^\succ\C{\textstyle\bigcap\limits_{R\in\Psi}}R^\succ.
\end{equation}
We immediately see that
\[\begin{matrix}\eqref{FiniteCMeetPreserving}&\Rightarrow&\eqref{CMeetPreserving}\\ &&\\ \Downarrow&&\Downarrow\\ &&\\ \eqref{FiniteMeetPreserving}&\Rightarrow&\eqref{MeetPreserving}\end{matrix}\]
Thus it suffices to prove
\[\eqref{MeetPreserving}\quad\Rightarrow\quad\eqref{Tight=>Filter}\quad\Rightarrow\quad\eqref{FiniteCMeetPreserving}.\]

For the first implication, say \eqref{MeetPreserving} holds and take tight $T\subseteq P$.  For any $p,q\in T=T^\prec$, we have $p',q'\in T$ with $p'\prec p$ and $q'\prec q$.  By \eqref{MeetPreserving}, $p'^\succ\cap q'^\succ\C p^\succ\cap q^\succ$.  As $T$ is tight, $p'^\succ\cap q'^\succ\not\C P\setminus T$ so $p^\succ\cap q^\succ\not\subseteq P\setminus T$, i.e. $T\cap p^\succ\cap q^\succ\neq\emptyset$ so $T$ contains a lower bound of $p$ and $q$.  As $p$ and $q$ were arbitrary, $T$ is downwards directed and hence a filter.

For the second implication, say \eqref{FiniteCMeetPreserving} fails, so we have non-empty finite $\Phi,\Psi\subseteq\mathscr{P}(P)$ such that, for all $R\in\Psi$, we have $Q\in\Phi$ with $Q\C R$ but
\begin{equation}\label{notC}
\bigcap_{Q\in\Phi}Q^\succ\not\C\bigcap_{R\in\Psi}R^\succ.
\end{equation}
By \eqref{CShrinking}, for each $R\in\Psi$ and $Q\in\Phi$ with $Q\C R$, we have a sequence $(F_n)$ of finite subsets such that, for all $n$,
\[Q\C F_{n+1}\prec F_n\prec R.\]
Note that
\begin{equation}\label{FnotC}
\Big(F_n^\succ\cap\bigcap_{S\in\Phi\setminus\{Q\}}S^\succ\Big)\not\C\bigcap_{R\in\Psi}R^\succ.
\end{equation}
Indeed, if this were not true then, as $Q\C F_n$ and hence $Q\D F_n$, \cite[Proposition 1.5 (1.5)]{BiceStarling2020HTight} would yield
\[\Big(\bigcap_{S\in\Phi}S^\succ\Big)\D\Big(F_n^\succ\cap\bigcap_{S\in\Phi\setminus\{Q\}}S^\succ\Big)\C \bigcap_{R\in\Psi}R^\succ.\]
But $U\D V\C W$ implies $U\C W$, by \cite[Proposition 1.5 (1.4)]{BiceStarling2020HTight}, so this would contradict \eqref{notC}.  As $F_n$ is finite, \eqref{FnotC} implies we have some $f_n\in F_n$ with
\[f_n^\succ\cap\bigcap_{S\in\Phi\setminus\{Q\}}S^\succ\not\C\bigcap_{R\in\Psi}R^\succ\]
(see \cite[Proposition 1.5 (1.7)]{BiceStarling2020HTight}).  By K\"onig's lemma for finitely branching trees (see \cite[Lemma III.5.6]{Kunen2011}) we can select $f_n\in F_n$ so that $(f_n)$ is decreasing, i.e. $f_{n+1}\prec f_n$, for all $n$.

Then we can again take $R'\in\Psi$ and $Q'\in\Phi$ with $Q'\C R'$ and $(F'_n)$ with
\[Q'\C F'_{n+1}\prec F'_n\prec R'.\]
As before, for all $m$ and $n$, we have
\[F_n'^\succ\cap f_m^\succ\cap\bigcap_{S\in\Phi\setminus\{Q,Q'\}}S^\succ\not\C\bigcap_{R\in\Psi}R^\succ,\]
As $(f_n)$ is decreasing, we can again select $f_n'\in F_n'$ such that, for all $m$,
\[f_n'^\succ\cap f_m^\succ\cap\bigcap_{S\in\Phi\setminus\{Q,Q'\}}S^\succ\not\C\bigcap_{R\in\Psi}R^\succ.\]
Yet again, K\"{o}nig's lemma means we can take $(f_n')$ decreasing.  Continuing in this way, we obtain a decreasing sequence for each element of $\Psi$.  Let $U$ be the union of these sequences, so $U\subseteq U^\prec$, $U^\prec\cap R\neq\emptyset$, for all $R\in\Psi$, and, for all finite $F\subseteq U$,
\[\widehat{F}\not\C\bigcap_{R\in\Psi}R^\succ\]
This means we can apply \cite[Theorem 2.11]{BiceStarling2020HTight} (with $Q$, $R$ and $S$ there replaced by $\emptyset$, $U$ and $\bigcap_{R\in\Psi}R^\succ$) to obtain tight $T$ containing $U$ but disjoint from $\bigcap_{R\in\Psi}R^\succ$.  But this means $T$ contains a finite subset $F$ which intersects every subset in $\Psi$, so any lower bound of $F$ would be in $\bigcap_{R\in\Psi}R^\succ$ and hence not in $T$.  Thus $T$ is not downwards directed and so not a filter.
\end{proof}

The \eqref{MeetPreserving}$\Rightarrow$\eqref{Tight=>Filter} part of Theorem \ref{MeetPreservingEquivalents} above generalises the fact that any maximal centred subset in a bi-pseudobasis must be an ultrafilter, as noted for posets in \cite[Proposition 1.4]{LawsonLenz2013}.  Even if $P$ is not a bi-pseudobasis, we can sometimes still show that the maximal centred filters are at least dense in the tight spectrum.

\begin{prp}\label{CountablePseudobasis}
If $P$ is a countable pseudobasis then the maximal centred filters are dense in $\mathcal{T}(P)=\{T\subseteq P:T\text{ is tight}\}$ with the topology generated by
\begin{equation}\label{TightBasis}
N^p_R=\{T\in\mathcal{T}(P):p\in T\text{ and }R\C P\setminus T\}.
\end{equation}
\end{prp}

\begin{proof}
If a filter is maximal centred then it is tight, by \cite[{}Proposition 2.9]{BiceStarling2020HTight}.  If $T\in N^p_R$ then we have $q\prec p$ with $q\perp R$ (otherwise $p\D R\C P\setminus T$ and hence $p\C P\setminus T$, contradicting the tightness of $T$).  Enumerate $P$ as a sequence $(p_n)$.  Set $q_1=q$ and recursively define a sequence $(q_n)$ as follows.  Having defined $q_n$, find the smallest $k$ such that $q_n\not\prec p_k$ and $q_n\not\perp p_k$.  Then just take any $q_{n+1}$ with $q_n,p_k\succ q_{n+1}$, noting that, as $q_{n+1}\prec p_k$, the next $k$ will be strictly larger (in the rare event that there was no such $k$ to begin with, the roundness of $P$ would imply $q_n\prec q_n$, in which case we could set $q_m=q_n$, for all $m>n$).  By construction, for every $p\in P$, we can find $q_n$ with $q_n\prec p$ or $q_n\perp p$.  Thus the upwards closure $U$ of $(q_n)$ will be maximal centred.  As $(q_n)$ is $\prec$-decreasing, $U$ will also be a filter.  As $q\in U$, $U\in N^q_\emptyset\subseteq N^p_R$.  As $N^p_R$ was an arbitrary non-empty basic open set, this shows that the maximal centred filters are indeed dense in $\mathcal{T}(P)$.
\end{proof}

Note that filters are determined by their initial segments, specifically
\[T\text{ is a filter}\qquad\Leftrightarrow\qquad\forall t\in T\ (T=(T\cap t^\succ)^\prec).\]
In bi-pseudobases, tightness is also determined by initial segments.

\begin{prp}
If $(P,\prec)$ is an abstract bi-pseudobasis and $\emptyset\neq T\subseteq p^\succ$ then
\begin{equation}\label{LocallyTight=>Tight}
T\text{ is tight in }p^\succ\qquad\Rightarrow\qquad T^\prec\text{ is tight in }P.
\end{equation}
\end{prp}

\begin{proof}
Say $T$ is tight in $p^\succ$.  By \eqref{Tight=>Filter}, $T$ is a filter in $p^\succ$ so $T^\prec$ is a filter in $P$.  Thus to prove that $T^\prec$ is tight in $P$ it suffices to show that
\begin{equation}\label{psucctinT}
p\succ t\in T\qquad\Rightarrow\qquad p\not\C P\setminus T^\prec.
\end{equation}
If $t\prec p\C P\setminus T^\prec$ then \eqref{CMeetPreserving} (with $Q=\{t\}$, $R=\{p\}$ and $S=P\setminus T^\prec$) yields
\[t\C p^\succ\cap(P\setminus T^\prec)^\succ\subseteq p^\succ\setminus T.\]
As $T$ is tight in $p^\succ$, this implies $t\notin T$, thus proving \eqref{psucctinT}.
\end{proof}

In other words, `locally' tight filters are automatically tight in bi-pseudobases.  This suggests that we might be able to generalise the theory of tight filters in bi-pseudobases by `localising', i.e. considering locally tight filters in local bi-pseudobases.  Indeed this is the case, as we show in the next section.  Moreover, this local generalisation is crucial if we want our theory to apply to all inverse semigroups with their canonical order \textendash\, see Example \ref{semilatticexpl} below.

But before moving on, let us just point out that in \eqref{LocalBiPseudobasis} we can use the given order $\prec$ rather than its lower preorder $\preceq$ if we already know that $P$ is an abstract pseudobasis.

\begin{prp}
An abstract local bi-pseudobasis is an abstract pseudobasis s.t.
\begin{equation}\label{LocallyMeetPreserving}
r'\prec r\prec p\quad\text{and}\quad s'\prec s\prec p\qquad\Rightarrow\qquad r'^\succ\cap s'^\succ\C r^\succ\cap s^\succ.
\end{equation}
\end{prp}

\begin{proof}
Taking $r=s=q$ and $r'=s'$ in \eqref{LocalBiPseudobasis}, we see that any abstract local bi-pseudobasis is an abstract pseudobasis.  Also \eqref{LocallyMeetPreserving} follows immediately from \eqref{LocalBiPseudobasis} and the fact that $\preceq$ weakens $\prec$ (because $\prec$ is transitive so $p\prec q$ implies $p^\succ\subseteq q^{\succ\succ}\subseteq q^\succ$).

Conversely, say $P$ is an abstract pseudobasis satisfying \eqref{LocallyMeetPreserving} and take $r,r',s,s',q$ with $r'\prec r\preceq q$ and $s'\prec s\preceq q$.  By \eqref{Shrinking}, we have finite $F$ with $r'\C F\prec r\preceq q$, and hence $F\prec q$.  Likewise, \eqref{Shrinking} yields finite $G$ with $s'\C G\prec s\preceq q$, and hence $G\prec q$.  By the \eqref{MeetPreserving} $\Rightarrow$ \eqref{CMeetPreserving} part of Theorem \ref{MeetPreservingEquivalents} applied within $q^\succ$,
\[r'^\succ\cap s'^\succ\C F^\succ\cap G^\succ\subseteq r^\succ\cap s^\succ.\]
Thus $r'^\succ\cap s'^\succ\C r^\succ\cap s^\succ$, showing that $P$ is an abstract local bi-pseudobasis.
\end{proof}

\begin{xpl}\label{semilatticexpl}
Say $(P,\leq)$ is a `generalised local $\wedge$-semilattice', i.e. a poset where every bounded pair $p,q\leq r$ is either disjoint $p\perp q$ or has an infimum $p\wedge q$, i.e.
\[p^\geq\cap q^\geq=\emptyset\qquad\text{or}\qquad p^\geq\cap q^\geq=(p\wedge q)^\geq.\]
As $p'\leq p$ and $q'\leq q$ implies $p'\wedge q'\leq p\wedge q$, it follows that $(P,\leq)$ is an abstract local bi-pseudobasis.

In particular, we can consider an inverse semigroup $S$ in its canonical ordering
\[p\leq q\qquad\Leftrightarrow\qquad p=pp^{-1}q.\]
As shown in \cite[\S1.4 Lemma 12 (3)]{Lawson1998}, bounded (or even compatible) pairs $p,q\in S$ always have meets, specifically $p\wedge q=pp^{-1}q$.  So if $S$ is an inverse semigroup with $0$ then $P=S\setminus\{0\}$ is an abstract local bi-pseudobasis with respect to the canonical order on $S$.  In this case our locally tight groupoid of $P$ will be the same as Exel's tight groupoid of $S$ (see the comments after Theorem \ref{GroupoidRecovery}).
\end{xpl}

\section{Locally Tight Filters}\label{LTF}

To turn an order structure $P$ into a topological space, on which the elements of $P$ get represented as open sets, one generally considers ultrafilters.  The motivation here is that, when $P$ is a basis of some locally compact space, the family of elements of $P$ containing a point $x$ do indeed form an ultrafilter w.r.t. compact containment.  As we are working with more general local bi-pseudobases, we need to consider more general `locally tight filters'.

\begin{center}
\textbf{Throughout this section, $(P,\prec)$ is an abstract local bi-pseudobasis}.
\end{center}

\begin{dfn}
We call $T\subseteq P$ \emph{locally tight} if $T\cap t^\succ$ is tight in $t^\succ$, for all $t\in T$.
\end{dfn}

By Theorem \ref{MeetPreservingEquivalents}, every locally tight $T\subseteq Q$ is a local filter, i.e. $T\cap t^\succ$ is a filter in $t^\succ$, for all $t\in T$.  Thus it suffices to consider singleton $F$ in Definition \ref{TightDefinition}.  So if $T=T^\prec$ then $T$ is locally tight iff, for all $s,t\in Q$,
\[t\succ s\in T\qquad\Rightarrow\qquad s\not\C t^\succ\setminus T.\]
More explicitly, $T=T^\prec$ is locally tight iff, for all $s,t\in Q$ and finite $F\subseteq Q$,
\begin{equation}\label{TightChars}
s\in t^\succ\cap T\quad\text{and}\quad F\prec t^\succ\setminus T\qquad\Rightarrow\qquad s^\succ\cap F^\perp\neq\emptyset.
\end{equation}

In contrast to our earlier work in \cite{BiceStarling2020HTight}, we will usually want $T$ itself to be a filter as well.  In this case, as in \eqref{LocallyTight=>Tight}, to verify local tightness it suffices to consider a single initial segment of $T$.

\begin{prp}\label{LocalTightness}
If $T$ is a filter in $P$ then, for any $t\in T$,
\[T\cap t^\succ\text{ is (locally) tight in }t^\succ\qquad\Leftrightarrow\qquad T\text{ is locally tight in }P.\]
\end{prp}

\begin{proof}
If $T$ is locally tight in $P$ then, by definition, $T\cap t^\succ$ tight (and hence locally tight) in $t^\succ$.  Conversely, say $T\cap t^\succ$ is locally tight in $t^\succ$.  For any $s\in T$, we have $r\in T$ with $r\prec s,t$, as $T$ is a filter.  As $r\prec t$, $T\cap r^\succ$ is tight in $r^\succ$.  As $r\prec s$ and $s^\succ$ is an abstract bi-pseudobasis, it follows from \eqref{LocallyTight=>Tight} (taking $r$ for $p$ and $s^\succ$ for $P$) that $T\cap s^\succ=(T\cap r^\succ)^\prec\cap s^\succ$ is tight in $s^\succ$.
\end{proof}

\begin{dfn}
The \emph{locally tight spectrum} $\mathcal{L}(P)$ is the set of all non-empty locally tight filters in $P$ with the topology generated by the sets $(O^p_R)_{R\,\C\,p}$ defined by
\[O^p_R=\{T\in\mathcal{L}(P):p\in T\text{ and }R\C p^\succ\setminus T\}.\]
\end{dfn}

\begin{rmk}
One should compare the locally tight spectrum $\mathcal{L}(P)$ with the tight spectrum $\mathcal{T}(P)$ from \cite[Definition 2.12]{BiceStarling2018} (see \eqref{TightBasis} above).  Intuitively, $\mathcal{T}(P)$ has more open sets than $\mathcal{L}(P)$ as there no $R\C p$ restriction.  This bounding restriction is important, just as it is in the order theoretic constructions of Paterson's universal groupoid and Exel's tight groupoid in \cite{LawsonLenz2013}.  Essentially, this is the reason why $\mathcal{T}(P)$ is always Hausdorff but $\mathcal{L}(P)$ is only locally Hausdorff.  However, these remarks are not really precise as $\mathcal{T}(P)$ and $\mathcal{L}(P)$ may not even have the same points \textendash\, $\mathcal{T}(P)$ can contain tight subsets that are not filters, while $\mathcal{L}(P)$ can contain filters that are not tight, only locally tight (also $R\C p^\succ\setminus T$ can be strictly stronger than $R\C P\setminus T$).  Although if $P$ is countable then $\mathcal{T}(P)$ and $\mathcal{L}(P)$ do at least share a dense subspace of maximal centred filters, by Proposition \ref{CountablePseudobasis} and Proposition \ref{UltrafilterDense} below (modifying the proof with the same enumeration argument).  And when $P$ is a bi-pseudobasis, the difference between $\mathcal{L}(P)$ and $\mathcal{T}(P)$ disappears \textendash\, see Theorem \ref{HausdorffEquivalents} below.
\end{rmk}

As $R\C p^\succ\setminus T$ means $R\D F\prec p^\succ\setminus T$, for some finite $F\subseteq P$, it would suffice to consider the smaller subbasis $(O^p_r)_{r\prec p}$ generating the topology of $\mathcal{L}(P)$.  On the other hand, to get a basis, we need to consider more general sets.  Let
\[O^G_F=\{T\in\mathcal{L}(P):G\subseteq T\text{ and }F\C G^\succ\setminus T\}.\]
Note that if $p\prec q$ then $p\in T\in\mathcal{L}(P)$ implies $\{p,q\}\subseteq T$, while $\{p,q\}^\succ=q^\succ$ so
\[O^{p,q}_F=\{T\in\mathcal{L}(P):p\in T\text{ and }F\C q^\succ\setminus T\}.\]

\begin{thm}\label{PairNeighbourhoodBase}
If $t\in T\in\mathcal{L}(P)$ then $(O^{p,t}_F)^{p,F\prec t}_{F\text{ is finite}}\ $ is a neighbourhood base at $T$.
\end{thm}

\begin{proof}
Say we are given a basic neighbourhood of $T$, i.e. $T\in\bigcap_{g\in G}O^g_{R_g}$ where $G$ is finite and $R_g\C g$, for each $g\in G$.  As $T$ is a filter, we have $p,q\in T$ with $t,G\succ q\succ p$.  For each $g\in G$, $T\in O^g_{R_g}$ so $R_g\C g^\succ\setminus T$ and \eqref{CShrinking} yields $Q_g$ with $R_g\C Q_g\prec g^\succ\setminus T$.  As we also have $p\prec q\prec g$, \eqref{CMeetPreserving} in $g^\succ$ yields $(p^\succ\cap Q_g^\succ)\C(q^\succ\cap(g^\succ\setminus T)^\succ)$.  By \eqref{CShrinking} again and the fact that $T^\prec\subseteq T$, we have finite $F_g$ with
\[p^\succ\cap Q_g^\succ\ \C\ F_g\ \prec\ q^\succ\cap(g^\succ\setminus T)^\succ\ \subseteq\ q^\succ\setminus T\ \subseteq\ t^\succ\setminus T.\]
In particular, $F_g\C t^\succ\setminus T$ so $T\in O^{p,t}_{F_g}$.

Setting $F=\bigcup_{g\in G}F_g$, we thus have $T\in O^{p,t}_F$ and we just have to show that $O^{p,t}_F\subseteq\bigcap_{g\in G}O^g_{R_g}$.  To see this, note that if $S\in O^{p,t}_F$ then $p\in S$ and hence $Q_g\cap S=\emptyset$ \textendash\, otherwise, as $S$ is a filter with $S\in O^{p,t}_F\subseteq O^t_{F_g}$, we would have $s$ with
\[s\in S\cap p^\succ\cap Q_g^\succ\C F_g\C t^\succ\setminus S,\]
contradicting the tightness of $S\cap t^\succ$ in $t^\succ$.  Thus $R_g\C Q_g\subseteq g^\succ\setminus S$ and hence $S\in O^g_{R_g}$, as $g\succ p\in S$.  As $S$ and $g$ were arbitrary, $O^{p,t}_F\subseteq\bigcap_{g\in G}O^g_{R_g}$.
\end{proof}

\begin{cor}\label{PairBasis}
$(O^{p,q}_F)^{p,F\prec q}_{F\text{ is finite}}\ $ is a basis for $\mathcal{L}(P)$.
\end{cor}

\begin{proof}
This is immediate from Theorem \ref{PairNeighbourhoodBase} and the fact that $O^{p,q}_F=O^p_\emptyset\cap O^q_F$ is open in $\mathcal{L}(P)$ whenever $p,F\prec q$ and $F$ is finite.
\end{proof}

Recall that an \emph{ultrafilter} is a maximal filter.  Every ultrafilter is locally tight and, moreover, every locally tight filter is a limit of ultrafilters.

\begin{prp}\label{UltrafilterDense}
Ultrafilters are dense in $\mathcal{L}(P)$.
\end{prp}

\begin{proof}
Say we have $p,q\in P$ and finite $F\subset P$ such that $p,F\prec q$ and $O^{p,q}_F\neq\emptyset$.  Taking $T\in O^{p,q}_F$, we have $F\C q^\succ\setminus T$ so \eqref{CShrinking} yields finite $G$ with $F\C G\prec q^\succ\setminus T$.  Note that we must have $p\not\D G$ \textendash\, otherwise we would have $p\D G\prec q^\succ\setminus T$ and hence $T\ni p\C q^\succ\setminus T$, contradicting the tightness of $T\cap q^\succ$ in $q^\succ$.  Thus we have some $r\in p^\succ\cap G^\perp$.

Let $U$ be any ultrafilter containing $r$.  As $U$ is a filter, $U\cap q^\succ$ must also be a maximal filter in $q^\succ$.  In fact, $U\cap q^\succ$ must be maximal even among round centred subsets of $q^\succ$, otherwise $U\cap q^\succ$ could be extended to a maximal round centred subset, which would be tight in $q^\succ$, by \cite[Proposition 2.9]{BiceStarling2020HTight}, and hence a filter in $q^\succ$, by Theorem \ref{MeetPreservingEquivalents}.  Thus $U\cap q^\succ$ is indeed tight in $q^\succ$, again by \cite[Proposition 2.9]{BiceStarling2020HTight}, and hence locally tight in $P$, by Proposition \ref{LocalTightness}.  Moreover $q\succ p\succ r\in U$ and $r\perp G$ so $F\C G\subseteq q^\succ\setminus U$.  Thus $U\in O^{p,q}_F$, proving density.
\end{proof}

Ultrafilters also have a nicer basis, namely $(O^p)_{p\in P}$ (where $O^p=O^p_\emptyset$).

\begin{prp}\label{UltrafilterNBase}
Every ultrafilter $U$ has $(O^u)_{u\in U}$ as a neighbourhood base.
\end{prp}

\begin{proof}
Say $(O^u)_{u\in U}$ is not a neighbourhood base of $U\in\mathcal{L}(P)$, so we have $p,F\prec t$ with $U\in O^{p,t}_F\nsupseteq O^u$, for all $u\in U$.  By \eqref{CShrinking}, we have $G$ with $F\C G\prec t^\succ\setminus U$.  For all $u\in U$ with $u\prec p,t$, we must have $u\not\perp G$, otherwise every $S\in O^u$ would be disjoint from $G$ and hence $F\C G\subseteq t^\succ\setminus S$ which would imply $O^u\subseteq O^{p,t}_F$, a contradiction.  But this means that $U\cap t^\succ$ is not maximal in $t^\succ$, by \cite[Proposition 2.9 (2.2)]{BiceStarling2020HTight}.  Thus we have a filter $V$ in $t^\succ$ properly extending $U\cap t^\succ$, which means $V^\prec$ is a filter properly extending $U$, i.e. $U$ is not an ultrafilter.
\end{proof}

At this point, one might wonder why we do not restrict our attention to ultrafilters.  One key reason is that ultrafilters on their own may not be locally compact.

\begin{xpl}
Consider the full countable infinitely branching tree or, more concretely, let $P=\mathbb{N}^{<\omega}$ be the poset of all finite sequences of natural numbers ordered by extension, i.e. $p\prec q$ means $p\supseteq q$.  Then the ultrafilters in the topology above can be identified with the infinite sequences $\mathbb{N}^\omega$ in their usual product topology, i.e. the Baire space.  This is homeomorphic to the irrationals in their usual subspace topology which is certainly not locally compact.

In contrast, every filter in $P$ is (locally) tight and these can be identified with arbitrary sequences, both finite and infinite, in a topology which makes them homeomorphic to the Cantor space.
\end{xpl}

Note that if we considered the full binary tree $\{0,1\}^{<\omega}$ instead then (subsequences of) finite sequences would no longer be (locally) tight, i.e. in this case every (locally) tight filter would be an ultrafilter.  More generally, locally tight filters coincide with ultrafilters precisely when $P$ satisfies a generalisation of the `trapping condition' from \cite{Lawson2012}.

\begin{thm}\label{TrappingEquivalents}
The following conditions on $P$ are equivalent.
\begin{align}
\label{OpBasis}(O^p)_{p\in P}\text{ is a basis for }&\mathcal{L}(P).\\
\label{LocallyTight=>Ultrafilter}T\text{ is a locally tight filter}\qquad\Rightarrow\qquad&T\text{ is an ultrafilter}.\\
\label{TrappingCondition}q'\prec q\prec p\quad\text{and}\quad r'\prec r\prec p\qquad\Rightarrow\qquad&q'^\succ\cap r^\perp\C q^\succ\cap r'^\perp.\\
\label{CTrappingCondition}Q'\C Q\prec p\quad\text{and}\quad R'\C R\prec p\qquad\Rightarrow\qquad&Q'^\succ\cap R^\perp\C Q^\succ\cap R'^\perp.
\end{align}
\end{thm}

\begin{proof}\
\begin{itemize}
\item[\eqref{CTrappingCondition}$\Rightarrow$\eqref{TrappingCondition}]  Immediate.

\item[\eqref{TrappingCondition}$\Rightarrow$\eqref{LocallyTight=>Ultrafilter}]
Assume \eqref{TrappingCondition} holds and take $p,q,q'\in T\in\mathcal{L}(P)$ with $q'\prec q\prec p$.  For any $F\C p^\succ\setminus T$, \eqref{CShrinking} yields finite $G$ and $G'$ with $F\C G'\prec G\prec p^\succ\setminus T$.  For every $g'\in G'$, we have $g\in G$ with $g'\prec g$ and hence \eqref{TrappingCondition} then yields $q'^\succ\cap G^\perp\subseteq q'^\succ\cap g^\perp\C q^\succ\cap g'^\perp$.  As $p^\succ$ is an abstract bi-pseudobasis, we obtain
\[q'^\succ\cap G^\perp\C\bigcap_{g'\in G'}(q^\succ\cap g'^\perp)^\succ\subseteq q^\succ\cap G'^\perp.\]
As $q'\in T$ and $G\prec p^\succ\setminus T$, \cite[{}Proposition 2.7]{BiceStarling2020HTight} in $p^\succ$ then yields
\[s\in T\cap q^\succ\cap G'^\perp\subseteq T\cap p^\succ\cap F^\perp.\]
As $F$ was arbitrary, $T\cap p^\succ$ is an ultrafilter in $p^\succ$, by \cite[{}Proposition 2.9 (2.2)]{BiceStarling2020HTight}, and hence $T$ is also an ultrafilter in $P$ (if not, then we could extend $T$ to an ultrafilter $U$, take $u\in U\setminus T$ and find $v\in U\setminus T$ with $v\prec u,p$, contradicting maximality in $p^\succ$).

\item[\eqref{LocallyTight=>Ultrafilter}$\Rightarrow$\eqref{OpBasis}]  See Proposition \ref{UltrafilterNBase}.

\item[\eqref{OpBasis}$\Rightarrow$\eqref{CTrappingCondition}]  Say \eqref{CTrappingCondition} fails so we have $Q'\C Q\prec p$ and $R'\C R\prec p$ with
\[Q'^\succ\cap R^\perp\not\C Q^\succ\cap R'^\perp.\]
By \eqref{CShrinking}, we have a sequence $(F_n)$ of finite subsets such that
\[Q'\C F_{n+1}\prec F_n\prec Q,\]
for all $n\in\mathbb{N}$.  As in the proof of Theorem \ref{MeetPreservingEquivalents}, K\"{o}nig's lemma yields $f_n\in F_n$ such that $f_{n+1}\prec f_n$ and $f_n^\succ\cap R^\perp\not\C Q^\succ\cap R'^\perp$ and hence $f_n^\succ\not\C(Q^\succ\cap R'^\perp)\cup R$, for all $n\in\mathbb{N}$.

Then \cite[Theorem 2.11]{BiceStarling2020HTight} (with $Q$, $R$ and $S$ replaced by $\emptyset$, $(f_n)$ and $(Q^\succ\cap R'^\perp)\cup R$) yields a tight subset of $p^\succ$ containing $(f_n)$ but disjoint from $(Q^\succ\cap R'^\perp)$ and $R$.  Taking the upwards closure yields a locally tight filter $T$ again containing $(f_n)$ but disjoint from $(Q^\succ\cap R'^\perp)$ and $R$.  Thus $R'\C R\subseteq p^\succ\setminus T$ so $T\in O^p_{R'}$.

If we also had some $s\in P$ with $T\in O^s\subseteq O^p_{R'}$ then, as $T$ is a filter, by replacing $s$ with a lower bound of $s$ and some $f_n$ in $T$ if necessary we may assume that $s\in Q^\succ$.  Also, we must have $s\perp R'$, otherwise any maximal filter $U$ containing some element of $s^\succ\cap R'^\succ$ would be in $O^s\setminus O^p_{R'}$, contradicting $O^s\subseteq O^p_{R'}$.  Thus $s\in Q^\succ\cap R'^\perp$, even though $T$ is disjoint from $Q^\succ\cap R'^\perp$, a contradiction.  Thus there is no such $s$ and hence $(O^s)_{s\in P}$ does not contain a neighborhood base for the point $T\in\mathcal{L}(P)$.
\end{itemize}
\end{proof}

Let $\Subset$ denote the compact containment relation in any topological space, i.e.
\[O\Subset N\qquad\Leftrightarrow\qquad\exists\text{ compact }C\ (O\subseteq C\subseteq N).\]
So in Hausdorff spaces, $O\Subset N$ is just saying that $\overline{O}$ is compact and $\overline{O}\subseteq N$.

\begin{thm}\label{qCrSubset}
Each $O^p$ is a Hausdorff subspace of $\mathcal{L}(P)$.  If $q,r\prec p$,
\[q\C r\qquad\Leftrightarrow\qquad O^q\Subset O^r.\]
\end{thm}

\begin{proof}
By Theorem \ref{PairNeighbourhoodBase}, $(O^{p,q}_F)^{q,F\prec p}_{F\text{ is finite}}$ forms a basis for $O^p$.  This combined with Theorem \ref{MeetPreservingEquivalents} in $p^\succ$ and Proposition \ref{LocalTightness} shows that $T\mapsto T\cap p^\succ$ is a homeomorphism from $O^p$ onto $\mathcal{T}(p^\succ)$.  But $p^\succ$ is an abstract pseudobasis so the result follows immediately from \cite[Proposition 2.17 and Theorem 2.26]{BiceStarling2020HTight}.
\end{proof}

Recall that a topological space is locally compact/Hausdorff if each point has a neighbourhood base of compact/Hausdorff subsets.  To be both locally compact and locally Hausdorff, it suffices for each point to have a single compact Hausdorff neighbourhood \textendash\, see \cite[Proposition 2.19]{BiceStarling2018}.

\begin{cor}\label{LCLH}
$\mathcal{L}(P)$ is locally compact and locally Hausdorff.
\end{cor}

\begin{proof}
Any $T\in\mathcal{L}(P)$ contains $p,q,r\in P$ with $q\prec r\prec p$.  By Theorem \ref{qCrSubset}, $T\in O^q\Subset O^r$ which means $T$ has a compact Hausdorff neighbourhood.
\end{proof}

In contrast to $\mathcal{T}(P)$, the entirety of $\mathcal{L}(P)$ is not always Hausdorff.  Indeed this happens precisely when $\mathcal{T}(P)$ and $\mathcal{L}(P)$ coincide, which also characterises bi-pseudobases among local bi-pseudobases.

\begin{thm}\label{HausdorffEquivalents}
The following are equivalent.
\begin{enumerate}
\item\label{LP=TP} $\mathcal{L}(P)=\mathcal{T}(P)$.
\item\label{LPHausdorff} $\mathcal{L}(P)$ is Hausdorff.
\item\label{Pbipseudo} $P$ is an abstract bi-pseudobasis.
\end{enumerate}
\end{thm}

\begin{proof}\
\begin{itemize}
\item[\eqref{LP=TP}$\Rightarrow$\eqref{LPHausdorff}]  By \cite[{}Proposition 2.17]{BiceStarling2020HTight}, $\mathcal{T}(P)$ is always Hausdorff.

\item[\eqref{LPHausdorff}$\Rightarrow$\eqref{Pbipseudo}]  If $P$ is not a bi-pseudobasis then we have non-filter $T\in\mathcal{T}(P)$, by Theorem \ref{MeetPreservingEquivalents}.  Take $p,q\in T$ with no lower bound in $T$.  As $T$ is tight, $T\cap p^\succ$ is tight in $p^\succ$ and also a filter in $p^\succ$, as $p^\succ$ is a bi-pseudobasis, so $U=(T\cap p^\succ)^\prec$ is a locally tight filter in $P$.  Likewise, $V=(T\cap q^\succ)^\prec$ is a locally tight filter, which is different from $U$ because $p\in U\setminus V$ and $q\in V\setminus U$.  Take any basic neighbourhoods $O^{s,p}_F$ of $U$ with $s,F\prec p$ and $F\C p^\succ\setminus U$, and $O^{t,q}_G$ of $V$ with $t,G\prec q$ and $G\C q^\succ\setminus V$.  Thus $s,t\in T$ and $F'\cup G'\prec P\setminus T$, where $F'$ and $G'$ are finite sets obtained from \eqref{CShrinking} which satisfy $F\C F'\prec p^\succ\setminus U\subseteq P\setminus T$ and $G\C G'\prec q^\succ\setminus V\subseteq P\setminus T$.  As $T$ is tight, we have $r\prec s,t$ with $r\perp F'\cup G'$ and hence any locally tight filter (e.g. ultrafilter) containing $r$ will lie in $O^{s,p}_F\cap O^{t,q}_G$.  Thus $U$ and $V$ can not be separated by disjoint open sets and hence $\mathcal{L}(P)$ is not Hausdorff.

\item[\eqref{Pbipseudo}$\Rightarrow$\eqref{LP=TP}]  If $P$ is an abstract bi-pseudobasis then every tight subset is a filter, by Theorem \ref{MeetPreservingEquivalents}, i.e. $\mathcal{T}(P)\subseteq\mathcal{L}(P)$.  But also every non-empty locally tight filter is tight, by \eqref{LocallyTight=>Tight}, i.e. $\mathcal{L}(P)\subseteq\mathcal{T}(P)$.  Thus as sets we have $\mathcal{L}(P)=\mathcal{T}(P)$.  Also, as $P$ is an abstract bi-pseudobasis, $F\C p$ implies $O^p_F=N^p_F$ so every open set in $\mathcal{L}(P)$ is open in $\mathcal{T}(P)$.  Conversely, again using the fact that the entirety of $P$ is an abstract bi-pseudobasis and arguing as in the proof of Theorem \ref{PairNeighbourhoodBase} shows that $(O^{p,q}_F)^{p,F\prec q}_{F\text{ is finite}}$ is a basis for $\mathcal{T}(P)$, i.e. $\mathcal{L}(P)=\mathcal{T}(P)$ as topological spaces.
\end{itemize}
\end{proof}

As mentioned in Example \ref{semilatticexpl}, the above results apply in particular when $P=S\setminus\{0\}$ and $\prec$ is the canonical ordering $\leq$ on an inverse semigroup or $\wedge$-semilattice $S$.  In this case, there are several corresponding results in the literature, namely
\begin{enumerate}
\item Corollary \ref{PairBasis} corresponds to \cite[Proposition 4.1]{Lenz2008}.
\item Proposition \ref{UltrafilterDense} corresponds to \cite[Theorem 12.9]{Exel2008}, \cite[Proposition 2.25]{Lawson2012} and \cite[Theorem 5.13]{LawsonLenz2013}.
\item Proposition \ref{UltrafilterNBase} corresponds to \cite[Proposition 2.5]{ExelPardo2016}.
\item Theorem \ref{TrappingEquivalents} corresponds to \cite[Proposition 2.27]{Lawson2012} and \cite[Theorem 5.19]{LawsonLenz2013}.
\item Theorem \ref{HausdorffEquivalents} corresponds to \cite[Proposition 5.2]{LawsonLenz2013}.
\end{enumerate}

Theorem \ref{HausdorffEquivalents} can even be used to recover the Hausdorff characterisation in \cite[Theorem 3.16]{ExelPardo2016}  \textendash\, being a bi-pseudobasis in this case means that, for all $p,q\in P$, $p^\geq\cap q^\geq\C p^\geq\cap q^\geq$.  In Exel's terminology, this is saying that $p^\geq\cap q^\geq$ has a finite cover which, as we are in an inverse semigroup, is the same as saying $(pp^{-1})^\geq\cap(qp^{-1})^\geq$ has a finite cover.  But this last set is just the collection $\mathcal{J}_{qp^{-1}}$ of idempotents below $qp^{-1}$, as in \cite[Theorem 3.16]{ExelPardo2016}.

\section{Concrete Local Bi-Pseudobases}\label{CLBP}

So far we know that any abstract local bi-pseudobasis $(P,\prec)$ can be represented as a collection of open sets $(O^p)_{p\in P}$ in its locally tight spectrum $\mathcal{L}(P)$.  While $(O^p)_{p\in P}$ will not always form a basis of $\mathcal{L}(P)$ (see Theorem \ref{TrappingEquivalents}), it will form a very similar slightly weaker kind of structure that we can axiomatise as follows.

\begin{dfn}
Let $\mathcal{O}(X)$ denote the collection of all non-empty open subsets of a topological space $X$.  We call $P\subseteq\mathcal{O}(X)$ a \emph{concrete local bi-pseudobasis} of $X$ if
\begin{align}
\label{LocallyHausdorff}\tag{Locally Hausdorff}&\text{Every $O\in P$ is a Hausdorff subspace of $X$}.\\
\label{Cover}\tag{Cover}&\text{Every $x\in X$ is contained in some $O\in P$}.\\
\label{PointFilter}\tag{Point-Filter}&\text{The neighborhoods in $P$ of any fixed $x\in X$ form a $\Subset$-filter}.\\
\label{Dense}\tag{Dense}&\text{Every $O\in\mathcal{O}(X)$ contains some $N\in P$}.\\
\label{Separating}\tag{Separating}&\text{The subsets in $P$ distinguish the points of $X$}.
\end{align}
\end{dfn}

Any basis of Hausdorff subsets of a locally compact space satisfies these conditions.  Conversely, any space with a concrete local bi-pseudobasis must be locally Hausdorff, by \eqref{LocallyHausdorff}, and locally compact, by \eqref{PointFilter}, i.e.
\[X\text{ has a concrete local bi-pseudobasis}\quad\Leftrightarrow\quad X\text{ is locally compact locally Hausdorff}.\]
Also, every concrete local bi-pseudobasis is a concrete pseudobasis in the sense of \cite[Definition 2.16]{BiceStarling2020HTight}.  The key extra conditions we are requiring are \eqref{LocallyHausdorff} and \eqref{PointFilter} (for pseudobases, the neighbourhoods in $P$ of any fixed $x\in X$ only have to be $\Subset$-round).

\begin{rmk}
While there is no Hausdorff requirement for each element of a pseudobasis in \cite[Definition 2.16]{BiceStarling2020HTight}, we can only recover $X$ from $\mathcal{T}(P)$ in \cite[{}Theorem 2.45]{BiceStarling2020HTight} when the entirety of $X$ is Hausdorff.  In contrast, even locally Hausdorff $X$ can be recovered from $\mathcal{L}(P)$ \textendash\, see Theorem \ref{Recovery} below.
\end{rmk}

\begin{prp}\label{abstractconcrete}
If $(P,\prec)$ is an abstract local bi-pseudobasis then $(O^p)_{p\in P}$ is a concrete local bi-pseudobasis of its locally tight spectrum $\mathcal{L}(P)$.
\end{prp}

\begin{proof}\
\begin{itemize}
\item[\eqref{LocallyHausdorff}]  See Theorem \ref{qCrSubset}.

\item[\eqref{Cover}]  As every $T\in\mathcal{L}(P)$ is non-empty, we have some $t\in T$ and hence $T\in O^t$.

\item[\eqref{PointFilter}]  If $T\in O^p\cap O^q$ then, as $T$ is a filter, we have $s,t\in T$ with $s\prec t\prec p,q$ and hence $T\in O^s\Subset O^t\subseteq O^p\cap O^q$, by Theorem \ref{qCrSubset}.

\item[\eqref{Dense}]  Any $O\in\mathcal{O}(\mathcal{L}(P))$ contains some non-empty $O^{p,t}_F$, where $p,F\prec t$ and $F$ is finite (see Corollary \ref{PairBasis}).  Take $T\in O^{p,t}_F$ so $p\in T$ and $F\C t^\succ\setminus T$.  By \eqref{CShrinking}, we have finite $G$ with $F\C G\prec t^\succ\setminus T$.  By \eqref{TightChars}, we have $q\in p^\succ\cap G^\perp$ so $O^q\subseteq O^{p,t}_F\subseteq O$.

\item[\eqref{Separating}]  If $T,U\in\mathcal{T}(P)$ are distinct then we have some $t\in T\setminus U$ or $t\in U\setminus T$, i.e. $U\notin O^t\ni T$ or $T\notin O^t\ni U$ so $O_t$ distinguishes $T$ from $U$.
\end{itemize}
\end{proof}

Conversely, any concrete local bi-pseudobasis is an abstract local bi-pseudobasis w.r.t. $\Subset$.  Moreover, we can recover the space from the locally tight spectrum, thus yielding a duality between concrete and abstract local bi-pseudobases.

\begin{thm}\label{Recovery}
If $P$ is a concrete local bi-pseudobasis of $X$ then $(P,\Subset)$ is an abstract local bi-pseudobasis.  Moreover, $X$ is homeomorphic to $\mathcal{L}(P)$ via the map
\[x\mapsto T_x=\{O\in P:x\in O\}.\]
\end{thm}

\begin{proof}
Certainly $\Subset$ is transitive.  Also $\Subset$ is round on $P$, by \eqref{PointFilter}.  It also follows from \eqref{PointFilter} that the lower preorder of $\Subset$ is just inclusion $\subseteq$.  By \eqref{LocallyHausdorff}, any $p\in P$ is a Hausdorff subspace so $r'\Subset r\subseteq p$ and $s'\Subset s\subseteq p$ means that $\overline{r'}\cap p$ and $\overline{s'}\cap p$ are compact and contained $r$ and $s$ respectively.  As intersections of compact subsets are again compact in Hausdorff spaces, $\overline{r'}\cap\overline{s'}\cap p$ is also compact and contained in $r\cap s$.  For each $x\in\overline{r'}\cap\overline{s'}\cap p$, \eqref{Cover} and \eqref{PointFilter} yield $p_x,q_x\in P$ with $x\in p_x\Subset q_x\Subset r\cap s$.  By compactness, we have a finite $Y\subseteq X$ such that $\overline{r'}\cap\overline{s'}\cap p\subseteq\bigcup_{x\in Y}p_x$.  Thus, by \eqref{PointFilter}, $r'^\Supset\cap s'^\Supset\D F=\{p_x:x\in Y\}$.  Also $G=\{q_x:x\in Y\}\subseteq(r\cap s)^\Supset\subseteq r^\Supset\cap s^\Supset$ and $F\subseteq G^\Supset$ which, by the definition of $\C$, means $r'^\Supset\cap s'^\Supset\C r^\Supset\cap s^\Supset$.  This shows that $(P,\Subset)$ satisfies \eqref{LocalBiPseudobasis}.

Now each $p\in P$ is a locally compact Hausdorff subspace of $X$ with concrete pseudobasis $p^\Supset$.  Thus $x\mapsto T_x\cap p^\Supset$ is homeomorphism from $p$ onto $\mathcal{T}(p^\Supset)$, by \cite[{}Theorem 2.45]{BiceStarling2020HTight}.  But as in noted in the proof of \eqref{qCrSubset}, $T\mapsto T\cap p^\Supset$ is a homeomorphism from $O^p\subseteq\mathcal{L}(P)$ onto $\mathcal{T}(p^\Supset)$.  Thus $x\mapsto T_x$ is a homeomorphism from $p$ onto $O^p$.  As $p$ was arbitrary, this shows that $x\mapsto T_x$ is a local homeomorphism.  But $(O^p)_{p\in P}$ covers $\mathcal{L}(P)$ so this map is also surjective.  By \eqref{Separating}, $x\mapsto T_x$ is also injective and hence a (total) homeomorphism.
\end{proof}

Again we can recover a previous result in the literature as a special case.  Specifically, when $P$ is a basis of compact open bisections (in which case $\Subset$ is just $\subseteq$) of an \'etale groupoid, Theorem \ref{Recovery} yields the topological part of \cite[Theorem 4.8]{Exel2010} (for the full generalisation of \cite[Theorem 4.8]{Exel2010}, see Theorem \ref{GroupoidRecovery} below).

\section{Ordered Groupoids}\label{OG}

Now we move on to a non-commutative extension of our construction by adding some groupoid structure, both to our topological spaces and our local bi-pseudobases.  More precisely, we obtain \'etale groupoids from ordered groupoids.

\begin{rmk}
Exel's original tight groupoid construction uses inverse semigroups, not ordered groupoids.  However, this would necessarily impose a $\wedge$-semilattice structure on the units/idempotents, even though we have not needed any $\wedge$-semilattice structure so far.  To avoid this restriction we use more general ordered groupoids instead, which also turn out to be more suitable for some of the examples we have in mind.  However, ordered groupoids and inverse semigroups are still closely related \textendash\, see Proposition \ref{IS=>OG} below and the comments after.
\end{rmk}

Recall that a \emph{groupoid} $G$ is a small category where all the morphisms are invertible.  As usual, we identify objects with their identity morphisms which we call \emph{units} and denote by $G^{(0)}$.  We also denote composable pairs in $G$ by $G^{(2)}$, i.e.
\begin{align*}
G^{(2)}&=\{(p,q):p,q\in G\text{ and }pq\text{ is defined}\}.\\
G^{(0)}&=\{p\in G:\forall q\in G\ ((p,q)\in G^{(2)}\ \Rightarrow\ pq=q=qp)\}.
\end{align*}
We also denote the range and source maps by $\mathsf{r}$ and $\mathsf{s}$, i.e.
\[\mathsf{r}(p)=pp^{-1}\qquad\text{and}\qquad\mathsf{s}(p)=p^{-1}p.\]

\begin{dfn}\label{OrderedGroupoid}
An \emph{ordered groupoid} $(G,\cdot,{}^{-1},\prec)$ is a groupoid $G$ together with a transitive relation $\prec$ on $G$ that preserves and reflects the formulas $pq$, $p^{-1}$ and $p^{-1}p$, i.e. for all $p,q,r\in G$ with $(p,q)\in G^{(2)}$,
\begin{align}
\label{Product}r&\prec pq&\Leftrightarrow&&\exists p'\prec p\ \exists q'\prec q\ &(r=p'q').\\
\label{Inverse}r&\prec p^{-1}&\Leftrightarrow&&\exists q\prec p\ &(r=q^{-1}).\\
\label{Support}r&\prec pp^{-1}&\Leftrightarrow&&\exists q\prec p\ &(r=qq^{-1}).
\end{align}
\end{dfn}

\begin{rmk}
As mentioned after \cite[\S4.1 Proposition 3]{Lawson1998}, it is important to note that when $G$ is a group this is much stronger than the usual notion \textendash\, ordered groups are usually just required to satisfy the $\Leftarrow$ part of \eqref{Product}.  If $\prec$ is also reflexive then the $\Rightarrow$ part of \eqref{Product} follows, but further imposing \eqref{Inverse} forces $\prec$ to be symmetric and hence an equivalence relation, i.e. a group congruence.  Further imposing \eqref{Support} then forces $\prec$ to be the the equality relation $=$ on the group $G$.
\end{rmk}

Note we are using `ordered' here in the broad sense of an arbitrary transitive relation, while the definition given \cite[\S4.1]{Lawson1998} refers only to partial orders.  In this case they agree, although the phrasing is somewhat different.  Indeed, using the essentially the same kind of arguments as in \cite[\S4.1 Proposition 3 and 4]{Lawson1998}, we can show that it suffices for preservation $(\Leftarrow)$ to hold in \eqref{Product} and reflection $(\Rightarrow)$ to hold in \eqref{Support} (while they amount to the same thing in \eqref{Inverse}).

\begin{prp}\label{OrderedGroupoidEquivalent}
An order $\prec$ makes a groupoid $G$ an ordered groupoid iff
\begin{align}
\label{Product'}p'\prec p\quad\text{and}\quad q'\prec q\qquad&\Rightarrow\qquad p'q'\prec pq\quad(\text{when }(p,q),(p',q')\in G^{(2)}).\\
\label{Inverse'}q\prec p\qquad&\Rightarrow\qquad q^{-1}\prec p^{-1}.\\
\label{Support'}r\prec pp^{-1}\qquad&\Rightarrow\qquad\exists q\prec p\ (r=qq^{-1}).
\end{align}
\end{prp}

\begin{proof}
As $p^{-1-1}=p$, \eqref{Inverse'} immediately yields \eqref{Inverse}.  Also \eqref{Support'} is the $\Rightarrow$ part of \eqref{Support}, while the $\Leftarrow$ part is immediate from \eqref{Product'} and \eqref{Inverse'}.  Likewise, \eqref{Product'} is the $\Leftarrow$ part of \eqref{Product}, so all we need to prove is the $\Rightarrow$ part.  To see this, say $r\prec pq$ so, by \eqref{Inverse'}, $r^{-1}\prec q^{-1}p^{-1}$ and hence, by \eqref{Product'}, $rr^{-1}\prec pqq^{-1}p^{-1}=pp^{-1}$.  By \eqref{Support'}, we have $p'\prec p$ with $rr^{-1}=p'p'^{-1}$.  Thus $p'^{-1}\prec p^{-1}$, by \eqref{Inverse'}, and $p'p'^{-1}r=rr^{-1}r=r\prec pq$ so, by \eqref{Product'}, $p'^{-1}r=p'^{-1}p'p'^{-1}r\prec p^{-1}pq=q$.  Setting $q'=p'^{-1}r$, we get $r=p'p'^{-1}r=p'q'$.  Thus $\prec$ does indeed reflect $pq$ so $G$ is an ordered groupoid.
\end{proof}

One thing \eqref{Support'} immediately tells us is that units are downwards closed, i.e.
\[G^{(0)\succ}\subseteq G^{(0)}.\]
Another thing to note is that the $q$ is \eqref{Support'} is unique.  Indeed, if $q,q'\prec p$ and $e=\mathsf{r}(q)=\mathsf{r}(q')$ then $(q^{-1},q')\in G^{(2)}$ and $q^{-1}q'\prec p^{-1}p$, so we have $p'\prec p$ with $q^{-1}q'=p'^{-1}p'$ and hence $q'=qq^{-1}q'=qp'^{-1}p=q$.  As in \cite[\S4.1]{Lawson1998}, we denote this `restriction' by ${}_e|p$, i.e. ${}_e|p$ is uniquely defined by
\[{}_e|p\prec p\qquad\text{and}\qquad\mathsf{r}({}_e|p)=e.\]
Likewise, if $e\prec p^{-1}p$ then we define $p|_e=({}_e|p^{-1})^{-1}$, so $p|_e$ (uniquely) satisfies
\[p|_e\prec p\qquad\text{and}\qquad\mathsf{s}(p|_e)=e.\]

Here are some basic results on sources of bounded subsets in ordered groupoids.

\begin{prp}
If $(G,\prec)$ is an ordered groupoid and $Q,R\prec p\in G$ then
\begin{gather}
\label{Q-1Q}\mathsf{s}[Q]=Q^{-1}Q=Q^{-1}p^\succ=p^{-1\succ}Q.\\
\label{setminusQ-1Q}\mathsf{s}[p^\succ\setminus Q]=\mathsf{s}(p)^\succ\setminus\mathsf{s}[Q].\\
\label{Q-1Qsucc}\mathsf{s}[Q^\succ]=\mathsf{s}[Q]^\succ.
\end{gather}
\begin{align}
\label{QperpR}Q\perp R\qquad&\Leftrightarrow\qquad\mathsf{s}[Q]\perp\mathsf{s}[R].\\
\label{QDR}Q\D R\qquad&\Leftrightarrow\qquad\mathsf{s}[Q]\D\mathsf{s}[R].\\
\label{QCR}Q\C R\qquad&\Leftrightarrow\qquad\mathsf{s}[Q]\C\mathsf{s}[R].
\end{align}
\end{prp}

\begin{proof}\
\begin{itemize}
\item[\eqref{Q-1Q}]  For any $q\in Q\prec p$ and $r\prec p$, if $(q^{-1},r)\in G^{(2)}$ then $\mathsf{r}(q)=\mathsf{r}(r)=e$ so $q=r={}_{e}|p$ and hence $q^{-1}r=q^{-1}q=\mathsf{s}(q)$.  Thus
\[Q^{-1}p^\succ\subseteq\mathsf{s}[Q]\subseteq Q^{-1}Q\subseteq Q^{-1}p^\succ.\]

\item[\eqref{setminusQ-1Q}] Certainly $\mathsf{s}[p^\succ\setminus Q]\subseteq\mathsf{s}[p^\succ]\subseteq\mathsf{s}(p)^\succ$.  Take $r\prec p$.  If $\mathsf{s}(r)=\mathsf{s}(q)=e$, for some $q\in Q\prec p$, then again $r=q=p|_e$, i.e. $\mathsf{s}(r)\in\mathsf{s}[Q]$ implies $r\in Q$ and hence $r\notin Q$ implies $\mathsf{s}(r)\notin\mathsf{s}[Q]$.  Thus $\mathsf{s}[p^\succ\setminus Q]\subseteq\mathsf{s}(p)^\succ\setminus\mathsf{s}[Q]$.  Conversely, if $e\in\mathsf{s}(p)^\succ\setminus\mathsf{s}[Q]$ then $\mathsf{s}(p|_e)=e\notin\mathsf{s}[Q]$ so $p|_e\notin Q$.  Thus $e=\mathsf{s}(p|_e)\in\mathsf{s}[p^\succ\setminus Q]$.

\item[\eqref{Q-1Qsucc}]  By \eqref{Product} and \eqref{Q-1Q},
\[\mathsf{s}[Q]^\succ\subseteq Q^{-1\succ}Q^\succ=Q^{\succ-1}Q^\succ=\mathsf{s}[Q^\succ]\subseteq\mathsf{s}[Q]^\succ.\]

\item[\eqref{QperpR}]  If $p\in Q^\succ\cap R^\succ$ then certainly $\mathsf{s}(p)\in\mathsf{s}[Q]^\succ\cap\mathsf{s}[R]^\succ$, proving $\Leftarrow$.  Conversely, if $e\prec\mathsf{s}(q)$ and $e\prec\mathsf{s}(r)$, for some $q\in Q$ and $r\in R$, then
\[p|_e=p|_{\mathsf{s}(q)}|_e\prec p|_{\mathsf{s}(q)}=q\qquad\text{and}\qquad p|_e=p|_{\mathsf{s}(r)}|_e\prec p|_{\mathsf{s}(r)}=r,\]
i.e. $p|_e\in q^\succ\cap r^\succ\subseteq Q^\succ\cap R^\succ$, which proves the $\Rightarrow$ part.

\item[\eqref{QDR}]  If $R\perp q\prec Q$ then $\mathsf{s}[R]\perp\mathsf{s}(q)\prec\mathsf{s}[Q]$, by \eqref{QperpR}.  Conversely, if $\mathsf{s}[R]\perp e\prec\mathsf{s}[Q]$ then \eqref{Q-1Qsucc} yields $e=\mathsf{s}(q)$, for some $q\prec Q$.  As $\mathsf{s}[R]\perp e=\mathsf{s}(q)$, \eqref{QperpR} again yields $R\perp q\prec Q$.

\item[\eqref{QCR}]  If $F$ is finite and $R\D F\prec Q$ then $\mathsf{s}[F]$ is finite and \eqref{QDR} yields $\mathsf{s}[R]\D\mathsf{s}[F]\prec\mathsf{s}[Q]$.  If $F'$ is finite and $\mathsf{s}[R]\D F'\prec\mathsf{s}[Q]$ then \eqref{Q-1Qsucc} yields finite $F\prec Q$ with $F'\subseteq\mathsf{s}[F]$ so \eqref{QDR} again yields $R\D F\prec Q$.
\end{itemize}
\end{proof}

Typical examples of ordered groupoids come from inverse semigroups.  Indeed, any inverse semigroup $S$ has a canonical partial order defined by
\[p\leq q\qquad\Leftrightarrow\qquad pp^{-1}=qp^{-1}\]
and by restricting the product $pq$ to the case $p^{-1}p=qq^{-1}$ we immediately obtain an ordered groupoid $(S,\leq)$ \textendash\, see \cite[\S3.1]{Lawson1998}.  More generally, we have the following.

\begin{prp}\label{IS=>OG}
Say $S$ is an inverse semigroup and $\prec$ is a transitive relation on $S$ strengthening $\leq$ and preserving the product and inverse operations, i.e.
\begin{align}
\label{Strengthening}p\prec q\qquad&\Rightarrow\qquad p\leq q.\\
\label{Product''}p\prec p'\quad\text{and}\quad q\prec q'\qquad&\Rightarrow\qquad pq\prec p'q'.\\
\label{Inverse''}p\prec q\qquad&\Rightarrow\qquad p^{-1}\prec q^{-1}.
\end{align}
Then $S^\succ$ is an ordered groupoid w.r.t. $\prec$ under the restricted product.
\end{prp}

\begin{proof}
By \eqref{Product''} and \eqref{Inverse''}, $S^\succ$ is closed under the product and inverse operations and hence forms a groupoid under the restricted product.  Also \eqref{Product''} and \eqref{Inverse''} immediately yield \eqref{Product'} and \eqref{Inverse'}, so the only thing left to prove is \eqref{Support'}.  So take $p,r\in S^\succ$ with $r\prec p^{-1}p$.  Thus we have some $q\in S$ with $p\prec q$ and hence $p\leq q$, by \eqref{Strengthening}, so $p=pp^{-1}p=qp^{-1}p\succ pr$, by \eqref{Product''}.  Again by \eqref{Strengthening}, $r\leq p^{-1}p$ so $r=r^{-1}r=r^{-1}p^{-1}pr$.  Thus we can take $p|_r=pr$, as required.
\end{proof}

\begin{rmk}
Typically, $S$ above will also have a $0$ which we will want to remove so that $(S,\prec)$ can form a non-trivial local bi-pseudobasis.
\end{rmk}

Conversely, given an ordered groupoid $(G,\prec)$, its downwards closed bisections
\[\mathcal{B}^\succ(G)=\{B\subseteq G:B^\succ\subseteq B\text{ and }B^{-1}B,BB^{-1}\subseteq G^{(0)}\}\]
are immediately seen to form an inverse semigroup.  The canonical order here is just inclusion $\subseteq$, which is strengthened by the transitive relation $\pp$ given by
\[B\pp C\qquad\Leftrightarrow\qquad\exists c\in C\ (B\prec c).\]
Moreover, $p\mapsto p^\succ$ maps $G$ to $\mathcal{B}^\succ(G)$ and preserves the ordered groupoid structure, i.e. for any $p,q\in G$, we have $p^{-1\succ}=p^{\succ-1}$, $(pq)^\succ=p^\succ q^\succ$ (if $(p,q)\in G^{(2)}$) and
\begin{equation}\label{psuccppqsucc}
p\prec q\qquad\Rightarrow\qquad p^\succ\pp q^\succ.
\end{equation}
This often allows us to identify $G$ with its image under this map, thus embedding the ordered groupoid $(G,\prec)$ into an ordered semigroup $(S,\pp)$.  However, in general there are some technicalities to consider, e.g. $p\mapsto p^\succ$ may not be injective and, even when it is, the converse to \eqref{psuccppqsucc} may fail.  Also, $\pp$ may fail to preserve the product in $\mathcal{B}^\succ(G)$, i.e. we can have $B\pp B'$ and $C\pp C'$ even though $BC\not\pp B'C'$.

Thus in general, ordered groupoids $(G,\prec)$ are significantly weaker than the `ordered inverse semigroups' $(S,\prec)$ appearing in Proposition \ref{IS=>OG}.  However, in the classical case when $\prec$ is a partial order $\leq$, these technicalities disappear and $(G,\leq)$ can always be embedded in $(\mathcal{B}^\geq(G),\subseteq)$ or the potentially smaller inverse semigroup $S$ it generates in $\mathcal{B}^\geq(G)$.  We can even have $S=G$, i.e. $G$ itself can be considered as inverse semigroup, namely when $G^{(0)}$ is a $\wedge$-semilattice, as shown in the Ehresmann-Schein-Nambooripad Theorem (see \cite[\S4.1 Theorem 8]{Lawson1998}).

\section{The Coset Groupoid}\label{TheCosetGroupoid}

Ultimately, we want to show that if a local bi-pseudobasis is also an ordered groupoid then that groupoid structure naturally passes to the locally tight filters.  First it is instructive to consider not just filters but more general `cosets'.

\begin{center}
\textbf{From now on $(G,\prec)$ is a fixed ordered groupoid.}
\end{center}

\begin{dfn}
We call $A\subseteq G$, \emph{unit-directed} if $\mathsf{r}[A]$ and $\mathsf{s}[A]$ are directed,
\begin{align*}
a,b\in A\quad\text{and}\quad\mathsf{s}(a)\succ\mathsf{s}(b)\qquad&\Rightarrow\qquad a|_{\mathsf{s}(b)}\in A,\qquad\text{and}\\
a,b\in A\quad\text{and}\quad\mathsf{r}(a)\succ\mathsf{r}(b)\qquad&\Rightarrow\qquad{}_{\mathsf{r}(b)}|a\in A.
\end{align*}

We call unit-directed $A\subseteq G$ an \emph{atlas} if $AA^{-1}A\subseteq A$.

We call an atlas $A\subseteq G$ a \emph{coset} if $A^\prec\subseteq A$.
\end{dfn}

These are analogous to the atlases and cosets defined for inverse semigroups in \cite{Lawson1998}.  Indeed, in inverse semigroups with their canonical order, unit-directedness already follows from $AA^{-1}A\subseteq A$ when interpreted with the original unrestricted product.  When we restrict the product to obtain a groupoid, unit-directedness has to be added as an explicit extra condition.

\begin{prp}\label{UniUp}
If $A$ is unit-directed then
\[A^\prec A^{\prec-1}\subseteq(AA^{-1})^\prec\qquad\text{and}\qquad(AA^{-1})^\prec A^\prec\subseteq(AA^{-1}A)^\prec.\]
Moreover, for any $a\in A$,
\begin{align}
\label{tTcap}\mathsf{s}[a^\succ\cap A^\prec]&=\mathsf{s}(a)^\succ\cap(A^{-1}A)^\prec.\\
\label{tTsetminus}\mathsf{s}[a^\succ\setminus A^\prec]&=\mathsf{s}(a)^\succ\setminus(A^{-1}A)^\prec.
\end{align}
\end{prp}

\begin{proof}
Say $a,b\in A^\prec$ and $ab^{-1}$ is defined.  Take $a',b'\in A$ with $a'\prec a$ and $b'\prec b$.  As $A$ is unit-directed, we can take $e\in\mathsf{s}[A]$ with $\mathsf{s}(a'),\mathsf{s}(b')\succ e$ and then $a|_e=a'|_e\in A$ and $b|_e=b'|_e\in A$.  As $ab^{-1}\succ a|_eb|_e^{-1}$, it follows that $ab^{-1}\in(AA^{-1})^\prec$.  As $a$ and $b$ where arbitrary, this shows that $A^\prec A^{\prec-1}\subseteq(AA^{-1})^\prec$.

Say $a\in(AA^{-1})^\prec$, $b\in A^\prec$ and $ab$ is defined.  Take $c,d\in A$ with $a\succ cd^{-1}$ and $b'\in A$ with $b'\prec b$.  As $A$ is unit-directed, we can take $e\in\mathsf{r}[A]$ with $\mathsf{r}(d),\mathsf{r}(b')\succ e$.  Then ${}_e|b={}_e|b'\in A$, ${}_e|d\in A$ and $c|_{\mathsf{s}({}_e|d)}\in A$.  It follows that $ab\succ c|_{\mathsf{s}({}_e|d)}d^{-1}|_e{}_e|b$, showing that $(AA^{-1})^\prec A^\prec\subseteq(AA^{-1}A)^\prec$.

\begin{itemize}
\item[\eqref{tTcap}]  As $A^{\prec-1}A^\prec\subseteq(A^{-1}A)^\prec$,
\[\mathsf{s}[a^\succ\cap A^\prec]\subseteq\mathsf{s}(a)^\succ\cap A^{\prec-1}A^\prec\subseteq\mathsf{s}(a)^\succ\cap(A^{-1}A)^\prec.\]
Conversely, if $e\in\mathsf{s}(a)^\succ\cap(A^{-1}A)^\prec$ then $b^{-1}c\prec e\prec a^{-1}a$, for some $b,c\in A$, and then $a|_e\succ a|_{\mathsf{s}(c)}\in A$ so $a_e\in a^\succ\cap A^\prec$ and $e=\mathsf{s}(a|_e)\in\mathsf{s}[a^\succ\cap A^\prec]$.

\item[\eqref{tTsetminus}]  Using \eqref{setminusQ-1Q} and \eqref{tTcap} we get
\begin{align*}
\mathsf{s}[a^\succ\setminus A^\prec]&=\mathsf{s}[a^\succ\setminus(a^\succ\cap A^\prec)]\\
&=\mathsf{s}(a)^\succ\setminus\mathsf{s}[a^\succ\cap A^\prec]\\
&=\mathsf{s}(a)^\succ\setminus(\mathsf{s}(a)^\succ\cap(A^{-1}A)^\prec)\\
&=\mathsf{s}(a)^\succ\setminus (A^{-1}A)^\prec.
\end{align*}
\end{itemize}
\end{proof}

As $\prec$ is transitive, $A^\prec A^{\prec-1}\subseteq(AA^{-1})^\prec$ yields $(A^\prec A^{\prec-1})^\prec\subseteq(AA^{-1})^\prec$.  As any unit-directed $A$ is round, i.e. $A\subseteq A^\prec$, the reverse inclusion is also immediate so
\[(A^\prec A^{\prec-1})^\prec=(AA^{-1})^\prec.\]
Likewise, for any unit directed $A$, we see that
\[(A^\prec A^{\prec-1}A^\prec)^\prec=((AA^{-1})^\prec A^\prec)^\prec=(A^\prec(A^{-1}A)^\prec)^\prec=(AA^{-1}A)^\prec.\]

\begin{prp}\label{AtlasUp}
If $A$ is an atlas then $A^\prec$ is a coset.
\end{prp}

\begin{proof}
As $\mathsf{r}[A]$ and $\mathsf{s}[A]$ are directed, so are $\mathsf{r}[A^\prec]\subseteq\mathsf{r}[A]^\prec$ and $\mathsf{s}[A^\prec]\subseteq\mathsf{s}[A]^\prec$.  Next say $a,b\in A^\prec$ and $\mathsf{s}(b)\prec\mathsf{s}(a)$.  Take $a',b'\in A$ with $a'\prec a$ and $b'\prec b$ and $e\in\mathsf{s}[A]$ with $\mathsf{s}(a'),\mathsf{s}(b')\succ e$.  Then $a|_{\mathsf{s}(b)}\succ a|_e=a'|_e\in A$ so $a|_{\mathsf{s}(b)}\in A^\prec$.  Likewise $\mathsf{r}(b)\prec\mathsf{r}(a)$ implies ${}_{\mathsf{r}(b)}|a\in A^\prec$.  So $A^\prec$ is unit-directed and hence $A^\prec A^{\prec-1}A^\prec\subseteq(AA^{-1}A)^\prec\subseteq A^\prec$, by Proposition \ref{UniUp}, i.e. as $A$ is an atlas, so is $A^\prec$.  As $\prec$ is transitive, $A^{\prec\prec}\subseteq A^\prec$ so $A^\prec$ is a coset.
\end{proof}

\begin{prp}\label{AtlasProduct}
If $g\in G$ and $A\subseteq G$ is an atlas then so is $Ag^\succ$.
\end{prp}

\begin{proof}
For any $a,b\in A$ and $j,k\prec g$ such that $aj$ and $bk$ are defined, we have $e\in\mathsf{r}[A]$ with $\mathsf{r}(a),\mathsf{r}(b)\succ e$.  Then ${}_e|a\in A$ and ${}_{\mathsf{s}({}_e|a)}|j\prec g$ so $e=\mathsf{r}({}_e|a{}_{\mathsf{s}({}_e|a)}|j)\in\mathsf{r}[Ag^\succ]$ so $\mathsf{r}[Ag^\succ]$ is directed.  We also have $f\in\mathsf{s}[A]$ with $\mathsf{s}(a),\mathsf{s}(b)\succ f$ and hence $\mathsf{r}(j),\mathsf{r}(k)\succ f$.  Thus $j,k\succ{}_f|g$ and hence $\mathsf{s}(j),\mathsf{s}(k)\succ\mathsf{s}({}_f|g)$ and $\mathsf{s}({}_f|g)=\mathsf{s}(a|_f{}_f|g)\in\mathsf{s}[Ag^\succ]$, showing that $\mathsf{s}[Ag^\succ]$ is also directed.  If $\mathsf{r}(b)=\mathsf{r}(bk)\prec\mathsf{r}(aj)=\mathsf{r}(a)$ then ${}_{\mathsf{r}(b)}|a\in A$ and ${}_{\mathsf{r}(bk)}|(aj)={}_{\mathsf{r}(b)}|a{}_{\mathsf{s}({}_{\mathsf{r}(b)}|a)}|g\in Ag^\succ$.  And if $\mathsf{s}(k)=\mathsf{s}(bk)\prec\mathsf{s}(aj)=\mathsf{s}(j)$ then $(aj)|_{\mathsf{s}(bk)}=a|_{\mathsf{s}(b)}k\in Ag^\succ$, showing that $Ag^\succ$ is unit-directed.  Moreover,
\[(Ag^\succ)(Ag^\succ)^{-1}(Ag^\succ)=Ag^\succ g^{\succ-1}A^{-1}Ag^\succ\subseteq A(gg^{-1})^\succ A^{-1}Ag^\succ\subseteq AA^{-1}Ag^\succ\subseteq Ag^\succ,\]
as $AA^{-1}A\subseteq A$, so $Ag^\succ$ is also an atlas.
\end{proof}

\begin{prp}\label{CosetUnitProduct}
If $A$ is a coset and $e\in\mathsf{s}[A]^\prec$ then
\[A=(Ae^\succ)^\prec\qquad\text{and}\qquad(AA^{-1})^\prec=(Ae^\succ A^{-1})^\prec.\]
\end{prp}

\begin{proof}
Certainly $(Ae^\succ)^\prec\subseteq(AG^{(0)})^\prec\subseteq A^\prec\subseteq A$.  Conversely, for any $a\in A$, we have $f\in\mathsf{s}[A]$ with $\mathsf{s}(a),e\succ f$ and then $a\succ a|_f=a|_ff\in Ae^\succ$, showing that $A\subseteq(Ae^\succ)^\prec$.  Now, for the second equality, just note that
\[(AA^{-1})^\prec=((Ae^\succ)^\prec(Ae^\succ)^{-1\prec})^\prec=(Ae^\succ e^\succ A^{-1})^\prec=(Ae^\succ A^{-1})^\prec,\]
by Proposition \ref{UniUp} and Proposition \ref{AtlasProduct}.
\end{proof}

\begin{prp}\label{CosetProducts}
If $A,B\subseteq G$ are cosets with $(A^{-1}A)^\prec=(BB^{-1})^\prec$ and $b\in B$,
\[\emptyset\neq Ab^\succ\subseteq AB\subseteq(Ab^\succ)^\prec.\]
\end{prp}

\begin{proof}
First note that $\mathsf{s}[A]^\prec=\mathsf{r}[B]^\prec$.  Indeed, if $e\in\mathsf{s}[A]\subseteq A^{-1}A\subseteq(BB^{-1})^\prec$, then we have $c,d\in B$ with $e\succ cd^{-1}$.  But then $cd^{-1}\in G^{(0)\succ} \subseteq G^{(0)}$ so $c=cd^{-1}d=d$ so $cd^{-1}\in\mathsf{r}[B]$ and hence $e\in\mathsf{r}[B]^\prec$.  This shows that $\mathsf{s}[A]\subseteq\mathsf{r}[B]^\prec$, while $\mathsf{r}[B]\subseteq\mathsf{s}[A]^\prec$ follows by a symmetric argument.  In particular, we have $a\in A$ with $\mathsf{s}(a)\prec\mathsf{r}(b)$ so $a{}_{\mathsf{s}(a)}|b\in Ab^\succ\neq\emptyset$.

Next note $a|_e\in A$ whenever $a\in A$, $e\prec\mathsf{s}(a)$ and $e\in\mathsf{r}[B]\subseteq\mathsf{s}[A]^\prec$.  Indeed, if $e\succ f\in\mathsf{s}[A]$ then $a|_e\succ a|_f\in A$ so $a|_e\in A^\prec\subseteq A$.  Likewise, ${}_e|b\in B$ whenever $b\in B$, $e\prec\mathsf{r}(b)$ and $e\in\mathsf{s}[A]$.

It follows that if $a\in A$, $c\prec b$ and $ac$ is defined then $c={}_{\mathsf{r}(c)}|b={}_{\mathsf{s}(a)}|b\in B$, showing that $Ab^\succ\subseteq AB$.

Conversely, take $a\in A$ and $c\in B$ where $ac$ is defined.  Further take $e\in\mathsf{s}[B]$ with $\mathsf{s}(b),\mathsf{s}(c)\succ e$.  Then $c|_e\in B$ so $a|_{\mathsf{r}(c|_e)}\in A$.  Moreover,
\[ac\succ a|_{\mathsf{r}(c|_e)}c|_e=a|_{\mathsf{r}(c|_e)}c|_e(b|_e)^{-1}b|_e\subseteq ABB^{-1}b^\succ\subseteq A^\prec(A^{-1}A)^\prec b^\succ\subseteq Ab^\succ,\]
as $A^\prec(A^{-1}A)^\prec\subseteq(AA^{-1}A)^\prec\subseteq A^\prec\subseteq A$.  Thus $ac\in(Ab^\succ)^\prec$, showing that $AB\subseteq(Ab^\succ)^\prec$.
\end{proof}

The above results can also be applied to the opposite groupoid, e.g. if $(A^{-1}A)^\prec=(BB^{-1})^\prec$ and $a\in A$ then Proposition \ref{CosetProducts} applied to the opposite groupoid yields
\[\emptyset\neq(a^\succ B)^\prec=(AB)^\prec.\]

Let $\mathcal{C}(G)$ denote the non-empty cosets, i.e.
\[\mathcal{C}(G)=\{A\subseteq G:A\neq\emptyset\text{ is a coset in }G\}.\]

\begin{thm}
$\mathcal{C}(G)$ forms a groupoid under the inverse $A\mapsto A^{-1}$ and product
\[(A,B)\mapsto(AB)^\prec\qquad\text{when}\qquad(A^{-1}A)^\prec=(BB^{-1})^\prec.\]
\end{thm}

\begin{proof}
Whenever $A,B\in\mathcal{C}(G)$, $(A^{-1}A)^\prec=(BB^{-1})^\prec$ and $a\in A$, Proposition \ref{AtlasUp}, Proposition \ref{AtlasProduct} and Proposition \ref{CosetProducts} yield $(AB)^\prec=(a^\succ B)^\prec\in\mathcal{C}(G)$, i.e. the product is well-defined.
Moreover, in this case Proposition \ref{CosetUnitProduct} yields
\[((AB)^{\prec-1}(AB)^\prec))^\prec=((a^\succ B)^{\prec-1}(a^\succ B)^\prec)^\prec=(B^{-1}a^{\succ-1}a^\succ B)^\prec=(B^{-1}B)^\prec,\]
as $a^{\succ-1}a^\succ=\mathsf{s}(a)^\succ$ and $\mathsf{s}[A]\subseteq\mathsf{r}[B]^\prec$.  So if $C\in\mathcal{C}(G)$ and $(B^{-1}B)^\prec=(CC^{-1})^\prec$ then the product of $(AB)^\prec$ and $C$ is also defined and is given by
\[((AB)^\prec C)^\prec=((ab)^\succ C)^\prec=(a^\succ b^\succ C)^\prec\subseteq(a^\succ BC)^\prec\subseteq(ABC)^\prec\subseteq((AB)^\prec C)^\prec\]
(for any $b\in B$).  Likewise, the product of $A$ and $(BC)^\prec$ is defined and is given by
\[(A(BC)^\prec)^\prec=(ABC)^\prec=((AB)^\prec C)^\prec,\]
i.e. the product is associative.  As $A\in\mathcal{C}(G)$, we immediately see that $A^{-1}\in\mathcal{C}(G)$ too.  Also Proposition \ref{CosetUnitProduct} yields $(A^{-1}AB)^\prec=(\mathsf{s}(a)^\succ B)^\prec=B$, as $\mathsf{s}[A]\subseteq\mathsf{r}[B]^\prec$, so $(A^{-1}A)^\prec$ is a unit in $\mathcal{C}(G)$.  Likewise, $(AA^{-1})^\prec$ is a unit in $\mathcal{C}(G)$ and hence $A^{-1}$ is indeed an inverse of $A$ in $\mathcal{C}(G)$. Thus $\mathcal{C}(G)$ is indeed a groupoid.
\end{proof}

\begin{prp}
$A\in\mathcal{C}(G)$ is a unit in $\mathcal{C}(G)$ iff $A\cap G^{(0)}\neq\emptyset$.
\end{prp}

\begin{proof}
If $A\in\mathcal{C}(G)$ is a unit then $\emptyset\neq\mathsf{r}[A]\subseteq(AA^{-1})^\prec\cap G^{(0)}=A\cap G^{(0)}$.  Conversely, if $e\in A\cap G^{(0)}\subseteq\mathsf{s}[A]$ then Proposition \ref{CosetUnitProduct} and Proposition \ref{CosetProducts} yield $A=(Ae^\succ)^\prec=(AA^{-1})^\prec$ so $A$ is a unit in $\mathcal{C}(G)$.
\end{proof}

We call $I\subseteq G$ an \emph{ideal} if, for all $g\in G$,
\[\tag{Ideal}\mathsf{r}(g)\in I\qquad\Leftrightarrow\qquad g\in I\qquad\Leftrightarrow\qquad\mathsf{s}(g)\in I.\]
Note every ideal is, in particular, a full subgroupoid, i.e.
\[\tag{Full Subgroupoid}g\in I\qquad\Leftrightarrow\qquad\mathsf{s}(g)\in I\quad\text{and}\quad\mathsf{r}(g)\in I.\]

Let $\mathcal{F}(G)$ denote the non-empty filters in $G$.

\begin{prp}
$\mathcal{F}(G)$ forms an ideal in $\mathcal{C}(G)$.
\end{prp}

\begin{proof}
First we show that every filter $A\subseteq G$ is a coset.  For any $a,b\in A$, we have $c\in A$ with $a,b\succ c$.  Then $\mathsf{s}(a),\mathsf{s}(b)\succ\mathsf{s}(c)$ and $\mathsf{r}(a),\mathsf{r}(b)\succ\mathsf{r}(c)$, showing that $\mathsf{s}[A]$ and $\mathsf{r}[A]$ are directed.  If $\mathsf{s}(b)\prec\mathsf{s}(a)$ too then $a|_{\mathsf{s}(b)}\succ a|_{\mathsf{s}(c)}=c\in A$ so $a|_{\mathsf{s}(b)}\in A^\prec\subseteq A$.  Likewise, $\mathsf{r}(b)\prec\mathsf{r}(a)$ implies ${}_{\mathsf{r}(b)}|a\in A$ so $A$ is unit-directed.  For any $a,b,c\in A$ such that $ab^{-1}c$ is defined, we can take $d\in A$ with $a,b,c\succ d$ and then $ab^{-1}c\succ dd^{-1}d=d\in A$.  Thus $ab^{-1}c\in A^\prec\subseteq A$, so $A$ is an atlas and hence a coset, as $A^\prec\subseteq A$.

Next note that if $A$ is directed then so is $Ag^\succ$, for any $g\in G$.  Indeed, if $a,b\in A$ and $aj$ and $bk$ are defined, for some $j$ and $k$ below $g$, then we can take $c\in A$ with $a,b\succ c$ and $aj,bk\succ c{}_{\mathsf{s}(c)}|g\in Ag^\succ$.  Thus if $A\in\mathcal{C}(G)$ and $(AA^{-1})^\prec\in\mathcal{F}(G)$ then, for any $a\in A$, Proposition \ref{CosetProducts} yields $A=(AA^{-1}A)^\prec=(AA^{-1}a^\succ)^\prec$ so $A$ is also directed and hence a filter.
\end{proof}

\section{The Locally Tight Groupoid}\label{LTG}

Now we can complete the non-commutative aspect of our locally tight groupoid construction, showing that the locally tight spectrum $\mathcal{L}(G)$ is not just locally compact but also an \'etale groupoid under the following standing assumption.

\begin{center}
\textbf{Let $(G,\prec)$ be an ordered groupoid and an abstract local bi-pseudobasis}.
\end{center}

First we show that single elements of $G$ naturally act on certain locally tight filters to produce another locally tight filter.

\begin{prp}\label{LocalAction}
If $T\in\mathcal{L}(G)$ and $\mathsf{s}(g)\in\mathsf{r}[T]^\prec$ then $(g^\succ T)^\prec\in\mathcal{L}(G)$.
\end{prp}

\begin{proof}
Take $t\in T$ with $\mathsf{r}(t)\prec\mathsf{s}(g)$ so $p=g|_{\mathsf{r}(t)}t\in(g^\succ T)^\prec$.  For any $q\prec p$, we have $q=g't'$ for $g'\prec g$ and $t'\prec t$.  By Proposition \ref{AtlasProduct}, $g^\succ T$ is unit-directed so
\[\mathsf{s}[p^\succ\setminus(g^\succ T)^\prec]=\mathsf{s}(p)^\succ\setminus((g^\succ T)^{-1}(g^\succ T))^\prec=\mathsf{s}(t)^\succ\setminus(T^{-1}T)^\prec=\mathsf{s}[t^\succ\setminus T],\]
by \eqref{tTsetminus} and Proposition \ref{CosetUnitProduct}.  Then \eqref{QCR} yields
\[q\C p^\succ\setminus(g^\succ T)^\prec\quad\Leftrightarrow\quad\mathsf{s}(t')=\mathsf{s}(q)\C\mathsf{s}[p^\succ\setminus(g^\succ T)^\prec]=\mathsf{s}[t^\succ\setminus T]\quad\Leftrightarrow\quad t'\C t^\succ\setminus T.\]
If $q\in(g^\succ T)^\prec$ then $t'=g'^{-1}q\in g'^{-1}(g^\succ T)^\prec\subseteq T^\prec=T$.  As $T\cap t^\succ$ is tight in $t^\succ$, we have $t'\not\C t^\succ\setminus T$ and hence $q\not\C p^\succ\setminus(g^\succ T)^\prec$.  This shows that $(g^\succ T)^\prec\cap p^\succ$ is tight in $p^\succ$ so $(g^\succ T)^\prec\in\mathcal{L}(G)$, by Proposition \ref{LocalTightness}.
\end{proof}

\begin{cor}
$\mathcal{L}(G)$ forms an ideal in $\mathcal{F}(G)$.
\end{cor}

\begin{proof}
It suffices to show that, for any $T\in\mathcal{F}(G)$,
\[T\in\mathcal{L}(G)\qquad\Leftrightarrow\qquad(T^{-1}T)^\prec\in\mathcal{L}(G).\]
For this all we need to do is take $t\in T$ and note that, by Proposition \ref{CosetProducts} and Proposition \ref{LocalAction}, $T\in\mathcal{L}(G)$ implies $(T^{-1}T)^\prec=(t^{-1\succ}T)^\prec\in\mathcal{L}(G)$ and conversely $(T^{-1}T)^\prec\in\mathcal{L}(G)$ implies $T=(T(T^{-1}T)^\prec)^\prec=(t^\succ(T^{-1}T)^\prec)^\prec\in\mathcal{L}(G)$.
\end{proof}

The only thing left to consider is the topology on $\mathcal{L}(G)$ and how it interacts with the groupoid structure.  Recall that a \emph{topological groupoid} is a groupoid together with a topology that makes the inverse and product (when it is defined) continuous.  An \emph{\'etale groupoid} is a topological groupoid such that the source map $\mathsf{s}$ is an open map (and hence a local homeomorphism \textendash\, see \cite[Theorem 5.18]{Resende2007}).

\begin{thm}\label{LGetale}
$\mathcal{L}(G)$ is a locally compact locally Hausdorff \'etale groupoid.
\end{thm}

\begin{proof}
By Corollary \ref{LCLH}, $\mathcal{L}(G)$ is locally compact and locally Hausdorff, we just need to show that $\mathcal{L}(G)$ is also \'etale.  First recall that $(O^{p,q}_F)^{p,F\prec q}_{F\text{ is finite}}\ $ forms a basis for $\mathcal{L}(G)$, by Corollary \ref{PairBasis}.  As $(O^{p,q}_F)^{-1}=O^{p^{-1},q^{-1}}_{F^{-1}}$, inversion is immediately seen to be continuous.  To see that the product is also continuous, take $T,U\in\mathcal{L}(G)$ with $(T^{-1}T)^\prec=(UU^{-1})^\prec$ and $p,F\prec q$ with $F$ finite and $(TU)^\prec\in O^{p,q}_F$.  By Theorem \ref{PairNeighbourhoodBase}, we can replace $O^{p,q}_F$ with a smaller basic open set if necessary to get $q=tu$, for some $t\in T$ and $u\in U$.  By \eqref{CShrinking}, we can take finite $F'$ with $F\C F'\prec q^\succ\setminus(TU)^\prec$.  As $\mathsf{s}(p)\in(TU)^{\prec-1}(TU)^\prec=(U^{-1}U)^\prec$ and $\mathsf{s}(p)\prec\mathsf{s}(q)=\mathsf{s}(u)$, \eqref{tTcap} yields $\mathsf{s}(p)\in\mathsf{s}[u^\succ\cap U]\subseteq\mathsf{s}[U]$ so $v=u|_{\mathsf{s}(p)}\in U$.  Let
\[G=\{t^{-1}|_{\mathsf{r}(f')}f':f'\in F'\}\prec t^{-1\succ}(q^\succ\setminus(TU)^\prec)\subseteq u^\succ\setminus U\]
(for the last inclusion, note that if we had $t'\prec t$ and $q'\prec q$ with $t'^{-1}q'\in U$ then $\mathsf{s}(t')=\mathsf{r}(t'^{-1}q')\in\mathsf{r}[U]\subseteq\mathsf{s}[T]^\prec$ so $t'=t|_{\mathsf{s}(t')}\in T$ and $q'=t't'^{-1}q'\in(TU)^\prec$).  Thus $U\in O^v_G$.  For any other $T'\in O^t$ and $U'\in O^v_G$ with $(T'^{-1}T')^\prec=(U'U'^{-1})^\prec$, Proposition \ref{CosetProducts} yields
\[(T'U')^\prec=(t^\succ U')^\prec\in O^{t|_{\mathsf{r}(v)}v}=O^{t|_{\mathsf{r}(v)}u|_{\mathsf{s}(p)}}=O^p.\]
We further claim that $(t^\succ U')^\prec\in O^q_F$.  As $F\C F'$, it suffices to show that $F'\subseteq q^\succ\setminus(t^\succ U')^\prec$.  To see this just note that if we had $f'\in F'\cap(t^\succ U')^\prec$ then we would have $t^{-1}|_{\mathsf{r}(f')}f'\in U'$ which, as $t^{-1}|_{\mathsf{r}(f')}f'\in G$, contradicts $U'\in O^v_G$.  Thus $(T'U')^\prec\in O^{p,q}_F$ which, as $T'$ and $U'$ were arbitrary, shows that $(TU)^\prec\in O^tO^v_G\subseteq O^{p,q}_F$.  This shows that the product is continuous.

To show that $T\mapsto(T^{-1}T)^\prec$ is an open map, it suffices to show that, for any finite $F$ with $p,F\prec q$,
\[\{(T^{-1}T)^\prec:T\in O^{p,q}_F\}=O^{\mathsf{s}(p),\mathsf{s}(q)}_{\mathsf{s}[F]}.\]
If $T\in O^{p,q}_F$ then $p\in T$ so certainly $\mathsf{s}(p)\in(T^{-1}T)^\prec$.  Also $F\C q^\succ\setminus T$ so
\[\mathsf{s}[F]\C\mathsf{s}[q^\succ\setminus T]=\mathsf{s}(q)^\succ\setminus(T^{-1}T)^\prec,\]
by \eqref{QCR} and \eqref{tTsetminus}.  Thus $(T^{-1}T)^\prec\in O^{\mathsf{s}(p),\mathsf{s}(q)}_{\mathsf{s}[F]}$.  Conversely, if $U\in O^{\mathsf{s}(p),\mathsf{s}(q)}_{\mathsf{s}[F]}$ then $p=pp^{-1}p\in(q^\succ U)^\prec\in\mathcal{L}(G)$, by Proposition \ref{LocalAction}.  Also \eqref{CosetUnitProduct} yields $U=(T^{-1}T)^\prec$, where $T=(q^\succ U)^\prec$, so \eqref{tTsetminus} yields
\[\mathsf{s}[F]\C\mathsf{s}(q)^\succ\setminus U=\mathsf{s}(q)^\succ\setminus(T^{-1}T)^\prec=\mathsf{s}[q^\succ\setminus T].\]
Thus $F\C q^\succ\setminus T$, by \eqref{QCR}, and hence $T\in O^{p,q}_F$, as required.
\end{proof}

\begin{rmk}
Originally, Exel only considered inverse semigroups $S$ with $0$ in their canonical order $\leq$, where the idempotents $E(S)$ necessarily form a $\wedge$-semilattice (see \cite[\S1.4 Proposition 8]{Lawson1998}).  Taking $(G,\prec)=(S\setminus\{0\},\leq)$, it follows that $G^{(0)}=E(S)\setminus\{0\}$ is a bi-pseudobasis so the unit space $\mathcal{L}(G)^{(0)}$ is necessarily Hausdorff, by Theorem \ref{HausdorffEquivalents}.  However, as we are considering more general ordered groupoids, even our unit space $\mathcal{L}(G)^{(0)}$ may not be Hausdorff, only locally Hausdorff.
\end{rmk}

We also immediately see that the representation $(O^p)_{p\in G}$ of $G$ in $\mathcal{L}(G)$ is not just a concrete local bi-pseudobasis, as shown in Proposition \ref{abstractconcrete}, but also a subgroupoid of the groupoid of all open bisections of $\mathcal{L}(G)$ as, whenever $\mathsf{s}(p)=\mathsf{r}(q)$,
\[(O^p)^{-1}=O^{p^{-1}}\qquad\text{and}\qquad O^pO^q=O^{pq}.\]
It only remains to show that we can reverse this process and recover an \'etale groupoid from a given subgroupoid of bisections, as in the following result.  This generalises \cite[Theorem 4.8]{Exel2010}, as the unit space $X^{(0)}$ is no longer required to be totally disconnected (or even Hausdorff and, moreover, instead of inverse semigroup bases, we consider more general ordered groupoid local bi-pseudobases).

\begin{thm}\label{GroupoidRecovery}
If $G$ is an ordered subgroupoid of bisections which is also a concrete local bi-pseudobasis of an \'etale groupoid $X$ then $X$ is isomorphic to $\mathcal{L}(G)$ via
\[x\mapsto T_x=\{O\in G:x\in O\}.\]
\end{thm}

\begin{proof}
By Theorem \ref{Recovery}, $x\mapsto T_x$ is a homeomorphism, we just need to show that it is also a groupoid isomorphism.  As $T_x^{-1}=T_{x^{-1}}$, we only need to show that
\begin{equation}\label{G2}
\mathsf{s}(x)=\mathsf{r}(y)\qquad\Leftrightarrow\qquad(T_x^{-1}T_x)^\Subset=(T_yT_y^{-1})^\Subset
\end{equation}
and, in this case, $T_{xy}=(T_xT_y)^\Subset$.  So take $x,y\in X$ with $\mathsf{s}(x)=\mathsf{r}(y)=z$.  For any $p\in T_x$ and $q\in T_y$, we have $\mathsf{s}[p],\mathsf{r}[q]\in T_z$.  By \eqref{PointFilter}, we have $e\in T_z$ with $\mathsf{s}[p],\mathsf{r}[q]\Supset e$.  As $G$ is an ordered subgroupoid of the groupoid of all bisections,
\[p|_e=\{w\in p:\mathsf{s}(w)\in e\}\in G\qquad\text{and}\qquad{}_e|q=\{w\in q:\mathsf{r}(w)\in e\}\in G.\]
and hence $p|_e\in T_x$ and ${}_e|q\in T_y$.  Then $\mathsf{s}[p]\Supset e=\mathsf{r}[{}_e|q]\in T_yT_y^{-1}$ and $\mathsf{r}[q]\Supset e=\mathsf{s}[p|_e]\in T_x^{-1}T_x$, i.e. $T_x^{-1}T_x\subseteq(T_yT_y^{-1})^\Subset$ and $T_yT_y^{-1}\subseteq(T_x^{-1}T_x)^\Subset$ and hence $(T_x^{-1}T_x)^\Subset=(T_yT_y^{-1})^\Subset$.

Conversely, if $\mathsf{s}(x)\neq\mathsf{r}(y)$ then \eqref{Separating} yields $e\in G^{(0)}$ which contains precisely one element of $\{\mathsf{s}(x),\mathsf{r}(y)\}$.  Say $\mathsf{s}(x)\notin e\ni \mathsf{r}(y)$.  By \eqref{Cover}, we have some some $q\in T_y$ and then \eqref{PointFilter} yields $f\in T_{\mathsf{r}(y)}$ with $e,\mathsf{r}[q]\Supset f$.  As $G$ is an ordered subgroupoid, ${}_f|q\in T_y$ so $f=\mathsf{r}[{}_f|q]\in(T_yT_y^{-1})^\Subset$ even though $f$ does not contain $\mathsf{s}(x)$ and so $f\notin(T_x^{-1}T_x)^\Subset$.  This completes the proof of \eqref{G2}.

If $G\ni p\Supset qr$, for some $q\in T_x$ and $r\in T_y$, then certainly $xy\in p$ and hence $p\in T_{xy}$.  Conversely, for any $p\in T_{xy}$, we have $\mathsf{r}[p]\in T_{\mathsf{r}(x)}$.  Taking any $q\in T_x$, we also have $\mathsf{r}[q]\in T_{\mathsf{r}(x)}$.  By \eqref{PointFilter}, we have some $e\in T_{\mathsf{r}(x)}$ with $\mathsf{r}[p],\mathsf{r}[q]\Supset e$.  As $G$ is an ordered subgroupoid, ${}_e|q\in T_x$, ${}_e|p\in T_{xy}$ and hence $q^{-1}|_e{}_e|p\in T_y$.  Thus $p\Supset{}_e|p={}_e|qq^{-1}|_e{}_e|p\in T_xT_y$ so $p\in(T_xT_y)^\Subset$, showing that $T_{xy}=(T_xT_y)^\Subset$.
\end{proof}

This finishes our locally tight groupoid construction.  The order theoretic approach taken here is in line with similar approaches to Paterson's universal groupoid in \cite{Lenz2008} and \cite{LawsonMargolisSteinberg2013} and Exel's tight groupoid in \cite{LawsonLenz2013}.  We feel this is the simplest approach, which also makes it clear that the topology on the resulting groupoid depends only on the order structure.  However, there is an alternative partial action approach which would be more in line with Paterson and Exel's original constructions (see \cite{ExelPardo2016}), as we briefly outline here.

Specifically, instead of considering the locally tight spectrum $\mathcal{L}(G)$ on the entirety of $G$, one restricts to the units $\mathcal{L}(G^{(0)})$.  For each $T\in\mathcal{L}(G^{(0)})$, an equivalence relation $\sim_T$ is defined on $\mathsf{s}^{-1}[T]$ by
\[g\sim_Th\qquad\Leftrightarrow\qquad\exists e\in T\ (g|_e=h|_e).\]
The \emph{germ} of $g$ at $T$ is $[g,T]=g^{\sim_T}\times\{T\}$.  To turn these germs into a groupoid, for each $g\in G$, a map $\beta_g$ from $O^{\mathsf{s}(g)}\subseteq\mathcal{L}(G^{(0)})$ onto $O^{\mathsf{r}(g)}\subseteq\mathcal{L}(G^{(0)})$ is defined by
\[\beta_g(T)=(g^\succ Tg^{-1\succ})^\prec.\]
The product of germs is then defined by
\[[g,T][h,U]=[gh,U]\qquad\text{when}\qquad T=\beta_h(U).\]
The topology on germs is generated by
\[\Theta(g,O)=\{[g,T]:T\in O\},\]
for open $O\subseteq O^{\mathsf{s}(g)}\subseteq\mathcal{L}(G^{(0)})$.  Then the resulting groupoid of germs $\mathcal{G}(G)$ is isomorphic to $\mathcal{L}(G)$ (both algebraically and topologically) via the map
\[[g,T]\mapsto(g^\succ T)^\prec\]
(to see that this map is a homeomorphism, note that it takes any basic open set $\Theta(g,O^{d,e}_F)\subseteq\mathcal{G}(G)$ to the basic open set $O^{g|_d,g|_e}_{(g|_f)_{f\in F}}\subseteq\mathcal{L}(G)$, while its inverse takes $O^{g,h}_F$, where $g,F\prec h$, to $\Theta(h,O^{\mathsf{s}(g),\mathsf{s}(h)}_{\mathsf{s}[F]})$).

\newpage

\bibliographystyle{spmpsci}
\bibliography{maths}

\end{document}